\documentclass[reqno,centertags,11pt,a4paper]{amsart}
\usepackage{amssymb,mathrsfs}
\usepackage{color,umoline}
\usepackage[dvipsnames]{xcolor}
\usepackage{graphicx}
\usepackage{enumitem}
\usepackage[font=small,labelfont=bf]{caption}
\usepackage{tikz, caption, subcaption}
\usepackage{relsize}
\usepackage[mathscr]{eucal}
\usepackage{verbatim}
\usepackage[draft]{todonotes}
\usepackage{hyperref}
\usepackage[dvipsnames]{xcolor}
\usetikzlibrary{patterns}
\tikzstyle{EDR}=[draw=lightgray,line width=0pt,preaction={clip, postaction={pattern=north east lines, pattern color=gray}}]
\tikzstyle{EDR1}=[draw=lightgray,line width=0pt,preaction={clip, postaction={pattern=north west lines, pattern color=gray}}]
\definecolor{mygray}{gray}{0.95}

\definecolor{mypink1}{rgb}{1.2,1.1,0.9}

\definecolor{mypink2}{rgb}{1.0,0.95 ,0.9}

\definecolor{mypink3}{rgb}{1.0,0.6,0.7}

\numberwithin{equation}{section}

\newtheorem{theorem}{Theorem}[section]

\newtheorem{lemma}[theorem]{Lemma}

\newtheorem{remark}[theorem]{Remark}

\numberwithin{equation}{section}

\newcommand{\beq}{\begin{equation}}
	\newcommand{\eeq}{\end{equation}}
\newcommand{\beqq}{\begin{equation*}}
	\newcommand{\eeqq}{\end{equation*}}
\newcommand{\ben}{\begin{eqnarray}}
	\newcommand{\een}{\end{eqnarray}}
\newcommand{\beno}{\begin{eqnarray*}}
	\newcommand{\eeno}{\end{eqnarray*}}

\begin{document}
%Insert here the title, affiliations and abstract:
\setcounter{page}{1}
\title[A noncommutative maximal ergodic inequality for arithmetic sets ]
{A noncommutative maximal inequality for ergodic averages along arithmetic sets}

\author[]{Cheng Chen}
\address{
Department of Mathematics\\
Sun Yat-sen(Zhongshan) University\\
Guangzhou 510275 \\
China}

\email{chench575@mail.sysu.edu.cn}

\author[]{Guixiang Hong}
\address{
Institute for Advanced Study in Mathematics\\
Harbin Institute of Technology\\
Harbin
150001\\
China}
\email{gxhong@hit.edu.cn}

%----------Author 2
\author[]{Liang Wang}
\address{
School of Mathematics and Statistics\\
Wuhan University\\
Wuhan 430072\\
China}

\email{wlmath@whu.edu.cn}

\thanks{}


\date{\today}
\begin{abstract}
In this paper, we establish a noncommutative maximal inequality for ergodic averages with respect to the set $\{k^t|k=1,2,3,...\}$ acting on noncommutative $L_p$ spaces for $p>\frac{\sqrt{5}+1}{2}$.
\end{abstract}
\maketitle
%\newline \indent $^{*}$Corresponding author%

\section{Introduction}

Let $(\Omega, \mu, T)$ be a dynamical system, where $\mu$ denotes a probability measure, $T$ stands for an invertible measure-preserving transformation.
The Birkhoff ergodic theorem indicates that for any $f\in L_1(\Omega)$, the ergodic average
\begin{equation*}
  A_Nf:=\frac{1}{N}\sum_{k=1}^{N}T^{k}f
\end{equation*}
converges almost everywhere as $N\rightarrow\infty$. It is commonly recognized that in numerous instances, establishing a maximal ergodic inequality suffices to derive a pointwise ergodic theorem. For instance, Wiener's  weak $(1,1)$ type estimate of the maximal operator associated to the time averages, as detailed in \cite{Wi}, could deduce the Birkhoff ergodic theorem.  In classical ergodic theory, the Birkhoff ergodic theorem, together with this strategy, was surrounded by numerous significant advancements. For instance,  Dunford and Schwartz \cite{DS} established the strong $(p,p)$ maximal inequalities for all $1<p<\infty$ in the context of the ergodic average associated with a positive $L_1$-$L_\infty$ contraction $T$. Subsequently, Akcoglu \cite{Ak} demonstrated that the validity of these strong $(p,p)$  maximal inequalities can be assured even when $T$ is merely a positive $L_p$ contraction. These results, viewed as ergodic theorems associated to the action of the group $\mathbb Z$, have been  extended to  a considerably large class of groups. For instance, after a series of progress made by Calder\'on, Wiener, Gromov among others, Lindenstrauss \cite{Li} further generalized the Birkhoff ergodic theorem to arbitrary amenable groups, broadening its scope and applicability. See the survey \cite{Nev06} by Nevo for ergodic theory associated to group actions. Motivated partially by the generalized uniform distribution problem, in his celebrated works \cite{Bo881,Bo882,Bo89}, Bourgain established the pointwise ergodic theorem for certain subset of integers by demonstrating the $L_p$-boundedness of maximal function for $p>1$. Specifically, Bourgain's results indicate that for any $f\in L_p(\Omega)$ with $p>1$, the ergodic average
\begin{equation*}
 A_Nf:=\frac{1}{N}\sum_{k=1}^NT^{k^t}f
\end{equation*}
converges almost everywhere as $N\rightarrow\infty$, and the maximal function $\sup_{N\geq1}|A_Nf|$ is $L_p$-bounded. Here $t\in\mathbb{N}^*$. Bourgain's work has been playing a central role in the development of the so-called discrete Fourier analysis, and one can find more information in the survey \cite{FITW20}[Sections II.1 and II. 3] and a recent work \cite{KMT22} as well as the references therein.

 On the other hand,  motivated by the quantum mechanics, operator algebra, noncommutative geometry, and noncommutative martingale theory, there appear more and more research papers on ergodic theory in the noncommutative setting. In the nascent stages of noncommutative ergodic theory, the focus was primarily on mean ergodic theorems. It was Lance's groundbreaking work \cite{La} that ignited the study of individual ergodic theorems, establishing noncommutative generalizations of the Birkhoff ergodic theorem for bounded functions. This topic then gained significant momentum, leading to a series of influential studies (see e.g. \cite{CoN,Ku}). Notably, after Yeadon's weak $(1,1)$ type maximal ergodic inequality \cite{Ye}, Junge and Xu \cite{JX07} obtained the full version of noncommutative analogue of the Dunford-Schwartz ergodic theorem. Their groundbreaking work has since inspired numerous subsequent studies on noncommutative ergodic theorems, see for instance \cite{CAD,Be,HRW,Hu,LS, HoSu18, LeXu12} and references therein. Beyond the class of Dunford-Schwartz operators explored by Junge and Xu, Hong \cite{Ho} studied the noncommutative ergodic averages of balls and spheres over Euclidean spaces. Furthermore, Hong, Liao and Wang \cite{HLW} obtained  noncommutative extensions of the Calder\'{o}n pointwise ergodic theorem for groups of polynomial growth. Recently, Cadilhac and Wang \cite{CW} established the maximal ergodic inequality for ergodic averages of amenable groups, marking a noncommutative analogue of the Lindenstrauss ergodic theorem. Despite these advancements, to the best of our knowledge, there remains a gap in the literature regarding a noncommutative analogue of Bourgain's ergodic theorem presented in \cite{Bo881,Bo882,Bo89}. In this paper, we delve into this unexplored territory, aiming to bridge this gap and contribute to the understanding of noncommutative ergodic theory by exploring the noncommutative analogue of Bourgain's seminal results.

Let $\mathcal{N}$ be a von Neumann algebra with a normal faithful semifinite trace $\tau$, and $\gamma$ be an automorphism on $\mathcal{N}$ satisfying $\tau=\tau\circ\gamma$.  The noncommutative $L_p$ spaces associated to $(\mathcal N,\tau)$ is denoted by $L_p(\mathcal N)$, then it is well-known that $\gamma$ extends to an isometry on $L_p(\mathcal N)$ for all $1\leq p\leq\infty$. We introduce the noncommutative
ergodic averages associated with $\gamma$ on $L_p(\mathcal{N})$ as follows,
\begin{equation}\label{ea1}
  A_Nx:=\frac{1}{N}\sum_{k=1}^N\gamma^{k^t}(x),\ x\in\mathcal{N},
\end{equation}
where $t\in\mathbb{N}^*$. Then we have the following noncommutative maximal ergodic inequality.

\begin{theorem}\label{thm1}
  For $p>\frac{\sqrt{5}+1}{2}$, the following inequality
  \begin{equation}\label{mthm1}
    \Big\|{\sup_{N}}^+A_N(x)\Big\|_{L_p(\mathcal{N})}\leq C_p\|x\|_{L_p(\mathcal{N})}
  \end{equation}
  holds for all $x\in{L_p(\mathcal{N})}$.
\end{theorem}

We refer the reader to the next section for the precise definition of the noncommutative maximal norm. Theorem \ref{thm1} extends Bourgain's result in \cite{Bo882} to the noncommutative setting. Extending these results to the noncommutative setting has been a challenging task, due to the inherent complexities and subtleties associated with noncommutative algebras.
Our approach to overcoming these difficulties was to establish a noncommutative analogue of the sampling principle, as outlined in \cite{MSW}. This principle allows us to adapt Bourgain's method to the noncommutative case, the successful application of this method not only validates the strength and generality of Bourgain's approach but also opens up new avenues for exploring ergodic theory in noncommutative settings.

\begin{remark}
 \rm{Up to now, due to the noncommutativity, it remains an outstanding challenge to obtain Theorem \ref{thm1} for all $p>1$ as in Bourgain's work \cite{Bo89}. The noncommutative framework introduces complexities that have hindered generalizations to the full range of $p$ values.

On the other hand, Bourgain's groundbreaking approach in \cite{Bo881, Bo882, Bo89} successfully utilized the quantitative maximal ergodic inequality to derive the pointwise ergodic theorem. This method circumvented the critical obstacle posed by the absence of a natural dense set of functions for which pointwise convergence holds. However, in the noncommutative setting, Bourgain's method faces a new hurdle: the need to tackle a distribution inequality. In the commutative case, distribution inequalities are well-understood and can be effectively employed. But in the noncommutative realm, dealing with such inequalities appears to be an extremely daunting task. The noncommutative structure introduces additional complexities that make it challenging to adapt Bourgain's method directly.

Consequently, obtaining a noncommutative pointwise ergodic theorem through the noncommutative maximal ergodic inequality presented in our paper remains an open problem. Despite the progress made in Theorem \ref{thm1}, further research is needed to bridge this gap and extend Bourgain's pointwise ergodic result to the noncommutative setting.}
%Such an achievement would be a significant milestone in noncommutative ergodic theory, with potential implications for various areas of mathematics and physics where noncommutative structures play a crucial role.
\end{remark}

The remainder of this paper is structured as follows. In Section \ref{pre}, we provide a review of noncommutative $L_p$ spaces, the noncommutative $\ell_\infty$-valued $L_p$ spaces, and some auxiliary results from \cite{Bo882}. This section serves as a foundation for the subsequent discussions, establishing notation and terminology used throughout the paper.

In Section \ref{sam}, we delve into the noncommutative realm, proving a noncommutative sampling principle (see Lemma \ref{sampling2}). This principle is a crucial component in our analysis, enabling us to extend Bourgain's ideas to the noncommutative setting.

Finally, in Section \ref{last}, we present the complete proof of Theorem \ref{thm1}. This theorem establishes the noncommutative maximal ergodic inequality, which is the main result of this paper. By leveraging the noncommutative sampling principle and other techniques, we demonstrate the validity of this inequality within the specified range of $p$.

 \textbf{Notation:}  In the subsequent discussion, $L_\infty(\Omega)\overline{\otimes}{\mathcal{M}}$  represents the tensor von Neumann algebra associated with a measurable space $(\Omega,\mu)$ and a semifinite von Neumann algebra $(\mathcal{M},\tau)$. The notation $A\lesssim B$ is used to indicate that there exists a positive constant $C$ such that $A\le C B$, and $A\sim B$ means that both $A\lesssim B$ and $B \lesssim A$ hold.
 Throughout the whole article, we denote by $C$ a positive constant which is independent of the main parameters, but it may
vary from line to line. For a given function $f$, $\Re f$ denotes the real part of $f$  while  $\Im f$ represents the imaginary part.
\section{Preliminaries}\label{pre}

\subsection{Noncommutative $L_p$ spaces}
 Let $(\mathcal{M},\tau)$ be a semifinite von Neumann algebra endowed with a normal semifinite faithful trace $\tau$.
 $\mathcal{S}(\mathcal{M})_+$ denotes the set of all $x\in\mathcal{M}_+$ such that $\tau(\mathrm{supp}(x))<\infty$, where $\mathrm{supp}(x)$ is the support of $x$. The linear span of $\mathcal{S}(\mathcal{M})_+$ is denoted by $\mathcal{S}(\mathcal{M})$ and constitutes  a weak*- dense *-subalgebra of $\mathcal{M}$. For $1\leq p<\infty$, the norm $\|\cdot\|_p$ is defined as \beqq \|x\|_p=(\tau(|x|^p))^{1/p},\quad x\in\mathcal{S}(\mathcal{M}),\eeqq
where $|x|=(x^*x)^{1/2}$ is the modulus of $x$.  This norm induces a completion of $\mathcal{S}(\mathcal{M})$ known as the noncommutative $L_p$ space associated to $(\mathcal{M},\tau)$, denoted by  $L_p(\mathcal{M})$.
As usual, we set $L_\infty(\mathcal{M})=\mathcal{M}$, and $\|x\|_\infty:=\|x\|_\mathcal{M}$.
For more information on noncommutative $L_p$ spaces, we refer to \cite{PX03}.

Frequently, we will make use of the operator-valued version of the H\"{o}lder inequality.
 \begin{lemma}\label{convex1}
Suppose that $f$ is a  $\mathcal{M}_+$-valued function on a measurable space $(\Sigma,\mu)$  and $g$ is a positive integrable function. Then for $1<p<\infty$, we have the following operator inequality,
 \begin{equation*}
  \int_{\Sigma}f(x)g(x)d\mu(x)\leq\Big(\int_{\Sigma}|g(x)|^{p^\prime}d\mu(x)\Big)^{1/{p^\prime}}\cdot\Big(\int_{\Sigma}|f(x)|^pd\mu(x)\Big)^{1/p},
 \end{equation*}
 whenever all the terms above make sense. Here, $p^\prime$ denotes the conjugate index of $p$, i.e. $\frac{1}{p}+\frac{1}{p^\prime}=1$.
 \end{lemma}
 This inequality is well-known in noncommutative analysis and can be found, for instance, in \cite[Lemma 2.4]{HLSX}.
\subsection{Noncommutative $\ell_\infty$-valued $L_p$ spaces}

Given $1\leq p\leq\infty$ and an index set $\Lambda$, the space $L_p(\mathcal{M};\ell_\infty(\Lambda))$ is the space of all $x=(x_\lambda)_{\lambda\in\Lambda}$ in $L_p(\mathcal{M})$ that can be factorized as:
\begin{equation*}
  x_\lambda=ay_\lambda b,\ \lambda\in\Lambda,
\end{equation*}
where $a,\ b\in L_{2p}(\mathcal{M}),\ y_\lambda\in L_\infty(\mathcal{M})$ and $ \ \sup_{\lambda\in\Lambda}\|y_\lambda\|_\infty<\infty$.
The norm in $L_p(\mathcal{M};\ell_\infty(\Lambda))$ is defined by
\begin{equation*}
  \| x\|_{L_p(\mathcal{M};\ell_\infty(\Lambda))}:=\inf_{x_\lambda=ay_\lambda b}\Big\{\|a\|_{2p}\sup_{\lambda\in\Lambda}\|y_\lambda\|_\infty\|b\|_{2p}\Big\}.
\end{equation*}
and this norm is commonly denoted by $\|{\sup_{\lambda\in\Lambda}}^+ x_\lambda\|_p$.
It was shown in \cite{JX07} that $x\in L_p(\mathcal{M};\ell_\infty(\Lambda))$ if and only if
\begin{equation*}
  \sup\Big\{\Big\|{\sup_{\lambda\in I}}^+ \ x_\lambda\Big\|_p\ :\  I\ \text{ is a finite subset of}\ \Lambda\Big\}<\infty.
\end{equation*}
In fact, $\|{\sup_{\lambda\in\Lambda}}^+ x_\lambda\|_p$ is equal to this supremum. Moreover, when $x=(x_\lambda)_{\lambda\in\Lambda}$ is a sequence of self-adjoint operators, then $x\in L_p(\mathcal{M};\ell_\infty(\Lambda))$ if and only if there exists a positive element $a\in L_p(\mathcal{M})$ such that for all $\lambda\in\Lambda$, $-a\leq x_\lambda\leq a$, and
\begin{equation*}
 \Big\|{\sup_{\lambda\in\Lambda}}^+\ x_\lambda\Big\|_p=\inf\Big\{\|a\|_p:\ a\in L_p(\mathcal{M})_+, -a\leq x_\lambda\leq a,\ \forall \lambda\in\Lambda\Big\}.
\end{equation*}
For brevity, when no confusion arises, we denote the space $L_p(\mathcal{M};\ell_\infty(\Lambda))$ simply by $L_p(\mathcal{M};\ell_\infty)$.

 Additionally, the following useful interpolation result can be found in \cite{JX07}.
\begin{lemma}\label{max interpolation}
 Suppose $1\leq p_0,p_1\leq\infty$ , $0<\theta<1$. Let $\frac{1}{p}=\frac{1-\theta}{p_0}+\frac{\theta}{p_1}$, then we have
      \begin{equation*}
        L_p(\mathcal{M};\ell_\infty)=\big(L_{p_0}(\mathcal{M};\ell_\infty), L_{p_1}(\mathcal{M};\ell_\infty)\big)_\theta.
      \end{equation*}
\end{lemma}

\subsection{Exponential sums and  approximate kernels}
Let $\mathbb{T}$ be the dual group of $\mathbb{Z}$ isomorphic to $[0,1)$ and let $K_N$ be the discrete measure defined as $K_N=\frac{1}{N}\sum_{k=1}^N\delta_{\{k^t\}}$. The Fourier transform of $K_N$ is given by the Guass-Weyl sum, expressed as
\begin{equation}\label{ese1}
  \hat{K}_N(\theta)=\frac{1}{N}\sum_{k=1}^Ne^{-2\pi ik^t\theta}, \quad \theta\in \mathbb{T}.
\end{equation}
Now, fix a constant $0<\rho<2$ and  let $\{p_j\}$ be the sequence of consecutive primes. For any $1\leq m\leq 2^{\rho s}, s\in\mathbb{N}^*$, one can define
\begin{equation*}
  Q_{s,m}=\prod_{(m-1)2^s<j\leq m2^s}p_j^{st},
\end{equation*}
and the set
\begin{equation*}
  \mathfrak{R}_s:=\{\theta\in \mathbb{T}|\ \theta \cdot Q_{s,1}\cdot Q_{s,m}\in \mathbb{Z},\ \text{for some}\  2\leq m\leq2^{\rho s}\}.
\end{equation*}
Observe that $\mathfrak{R}_s$ can be rewritten as
\begin{equation}\label{preq1}
  \mathfrak{R}_s= \mathbb{Z}_{Q{s,1}}\cup\Big(\bigcup_{ 2\leq m\leq2^{\rho s}}( \mathbb{Z}_{Q{s,1}\cdot Q{s,m}}\backslash \mathbb{Z}_{Q{s,1}})\Big),
\end{equation}
where
\begin{equation*}
  \mathbb{Z}_Q:=\{a/Q: 0\leq a<Q\}.
\end{equation*}
Then $\{\mathfrak{R}_s\}$ is an increasing sequence whose union is $\mathbb{T}\cap \mathbb{Q}$ and each $\mathfrak{R}_s$ is the disjoint union of $2^{\rho s}$ differences of cyclic subgroups of $\mathbb{T}$. For $\xi\in \mathbb{T}\cap \mathbb{Q}$, expressed as $\xi=a/q$ with $(a,q)=1$,  one can define
\begin{equation*}
  S(\xi):=\frac{1}{q}\sum_{0\leq r<q}e^{-2\pi i r^ta/q}.
\end{equation*}
There exists an estimate for $S(\xi)$ (see e.g. \cite[Lemma 4]{Bo882}).
\begin{lemma}\label{akeq4}
  If $a/q\in \mathbb{T}\backslash \mathfrak{R}_s$, then $|S(a/q)|\leq C 2^{-(1+\rho)s/2}.$
\end{lemma}

One can define an approximate kernel $L_N$ whose Fourier transform is expressed as
\begin{equation}\label{akeq1}
   \hat{L}_N(\theta)=\hat{k}(N^t\theta)\phi(\theta)+\sum_{s=1}^\infty\sum_{\xi\in \mathfrak{R}_s\backslash\mathfrak{R}_{s-1}}S(\xi)\hat{k}(N^t(\theta-\xi))\phi(D_s(\theta-\xi)).
\end{equation}
Here, $k(x):=t^{-1}x^{1/t-1}\chi_{[0,1]}(x)$ is the kernel function, and $D_s$ is an integer depending on $s$ that satisfies the conditions
\begin{equation*}
  \log Q_{s,m}<\frac{1}{100}\log D_s\ (1\leq m\leq 2^{s\rho}) \quad \text{and} \ \log D_s\leq Cs^22^s.
\end{equation*}
The function $\phi$ is a smooth function on $\mathbb{R}$ that equals $1$ on $[\frac{-1}{10},\frac{1}{10}]$ and vanishes outside $[\frac{-1}{2},\frac{1}{2}]$.

Two crucial estimates for the approximate kernels are provided in the literature (see \cite[Lemma 5, Lemma 6]{Bo882} for details).
\begin{lemma}\label{akey3}
  For $0<\rho^\prime<\rho$, one has
  \begin{equation*}
    \|\hat{K}_N-\hat{L}_N\|_\infty\leq C(\log N)^{-(1+\rho^\prime)/2}.
  \end{equation*}
\end{lemma}
\begin{lemma}\label{akey2}
  For $q< D_s$, there is a uniform estimate
  \begin{equation*}
    \Big\|\sum_{0\leq a <q }\int_\mathbb{T} S(a/q)\hat{k}(N^t(\theta-a/q))\phi(D_s(\theta-a/q))e^{2\pi i\theta m} d\theta\Big\|_{\ell_1(\mathbb{Z})}\leq C.
  \end{equation*}
  Moreover, one has
   \begin{equation*}
    \Big\|\sum_{\xi\in \mathfrak{R}_s }\int_\mathbb{T} S(\xi)\hat{k}(N^t(\theta-\xi))\phi(D_s(\theta-\xi))e^{2\pi i\theta m} d\theta\Big\|_{\ell_1(\mathbb{Z})}\leq C2^{\rho s}
    \end{equation*}
   and
   \begin{equation*}
    \Big\|\sum_{\xi\in \mathfrak{R}_{s-1} }\int_\mathbb{T} S(\xi)\hat{k}(N^t(\theta-\xi))\phi(D_s(\theta-\xi))e^{2\pi i\theta m} d\theta\Big\|_{\ell_1(\mathbb{Z})}\leq C2^{\rho s}.
    \end{equation*}
\end{lemma}
\begin{remark}
 \rm{ It is worth pointing out that the reference \cite{Bo882} assumes that $0<\rho<1$, however, we find that we need some $\rho>1$ to make sure  Theorem \ref{thm1} holds for all $p>\frac{1+\sqrt{5}}{2}$, see the equality \eqref{finaleq} for more details. Therefore, we assume that $0<\rho<2$, which is sufficient to guarantee the validity of all the estimates required in this subsection. Moreover, we believe that even the reference \cite{Bo882} should assume that $0<\rho<2$.}
\end{remark}

\section{A noncommutative sampling principle}\label{sam}
In this section, we aim to establish a noncommutative analogue of the sampling principle stated in \cite{MSW}.
Given the function $\Phi(x):=\Big(\frac{\sin(\pi x)}{\pi x}\Big)^2$ defined for $ x\in\mathbb{R}$. Then for a suitable operator-valued function $f$ on $\mathbb{Z}$, we define its extension $f_{ext}$ on $\mathbb{R}$ by
\begin{equation}\label{ext1}
  f_{ext}(x):=\sum_{n\in \mathbb{Z}}f(n)\Phi(x-n).
\end{equation}
Note that the series above converges for every $x\in\mathbb{R}$ whenever $f$ belongs to $L_p(L_\infty(\mathbb{Z})\overline{\otimes}\mathcal{M})$ for $1\leq p\leq\infty.$ Furthermore, we have the following equivalence.
\begin{lemma}\label{extnorm}
For $1\leq p\leq\infty$, if $f\in L_p(L_\infty(\mathbb{Z})\overline{\otimes}\mathcal{M})$, then we have
\begin{equation*}
  \|f\|_{L_p(L_\infty(\mathbb{Z})\overline{\otimes}\mathcal{M})}\sim  \|f_{ext}\|_{L_p(L_\infty(\mathbb{R})\overline{\otimes}\mathcal{M})}.
\end{equation*}
\end{lemma}
\begin{proof}
 Without loss of generality, we can assume that $f$ is positive. Given that $1\leq p<\infty$. First, recognizing that both $f$ and $\Phi$ are positive,  we fix $x\in \mathbb{R}$, by Lemma \ref{convex1}, we have
 \begin{equation*}
   f_{ext}(x)\leq\Big(\sum_{n\in \mathbb{Z}}|f(n)|^p\Phi(x-n)\Big)^{1/p}\Big(\sum_{n\in \mathbb{Z}}\Phi(x-n)\Big)^{(p-1)/p}.
 \end{equation*}
 Then, combining with the facts (see \cite[page 193]{MSW})
 \begin{equation*}
   \int_{\mathbb{R}}\Phi(x)dx\lesssim1, \qquad \sup_{x\in\mathbb{R}}\sum_{n\in \mathbb{Z}}\Phi(x-n)\lesssim1,
 \end{equation*}
we can further deduce that
 \begin{align*}
    \|f_{ext}\|_{L_p(L_\infty(\mathbb{R})\overline{\otimes}\mathcal{M})}^p&\leq \tau\Big( \int_{\mathbb{R}}\Big(\sum_{n\in \mathbb{Z}}|f(n)|^p\Phi(x-n)\Big)\Big(\sum_{n\in \mathbb{Z}}\Phi(x-n)\Big)^{p-1}dx\Big)\\
&\leq \tau\Big( \sum_{n\in \mathbb{Z}}|f(n)|^p\Big)\times\Big(\int_{\mathbb{R}}\Phi(x)dx\Big)\times \Big(\sup_{x\in\mathbb{R}}\sum_{n\in \mathbb{Z}}\Phi(x-n)\Big)^{p-1}\\
&\lesssim \|f\|_{L_p(L_\infty(\mathbb{Z})\overline{\otimes}\mathcal{M})}^p.
 \end{align*}
 On the other hand, let $\Psi$ be a smooth function on $\mathbb{R}$ such that its Fourier transform $\widehat{\Psi}$ equals $1$ on $[-1,1]$ and has compact support. One can see e.g. \cite[ page 194]{MSW} that
  \begin{equation*}
   \int_{\mathbb{R}}|\Psi(x)|dx\lesssim1, \qquad \sup_{x\in\mathbb{R}}\sum_{n\in \mathbb{Z}}|\Psi(n-x)|\lesssim1
 \end{equation*}
 and
 \begin{equation*}
   f(n)=\int_{\mathbb{R}}f_{ext}(x)\Psi(n-x)dx, \quad n\in \mathbb{Z}.
 \end{equation*}
 Note that $\Psi$ can be expressed as a linear combination of four positive functions $ \Psi_i$ for $1\leq i\leq4$ that satisfy
  \begin{equation*}
   \int_{\mathbb{R}}\Psi_idx\lesssim1, \qquad \sup_{x\in\mathbb{R}}\sum_{n\in \mathbb{Z}}\Psi_i(n-x)\lesssim1.
 \end{equation*}
 Precisely, we define $\Psi_1=\big(\Re(\Psi)\big)_+, \Psi_2=\big(\Re(\Psi)\big)_-, \Psi_3=\big(\Im(\Psi)\big)_+, \Psi_4=\big(\Im(\Psi)\big)_-.$
 Let $f_i(n):=\int_{\mathbb{R}}f_{ext}(x)\Psi_i(n-x)dx$, by invoking Lemma \ref{convex1}, we obtain
  \begin{equation*}
   f_i(n)\leq\Big(\int_{\mathbb{R}}|f_{ext}(x)|^p\Psi_i(n-x)dx\Big)^{1/p}\Big(\int_{\mathbb{R}}\Psi_i(n-x)dx\Big)^{(p-1)/p}.
 \end{equation*}
 Then we have
 \begin{align*}
    \|f_i\|_{L_p(L_\infty(\mathbb{Z})\overline{\otimes}\mathcal{M})}^p= \tau\Big( \sum_{n\in \mathbb{Z}}|f_i(n)|^p\Big)\lesssim \tau\Big( \sum_{n\in \mathbb{Z}}\int_{\mathbb{R}}|f_{ext}(x)|^p\Psi_i(n-x)dx\Big)\lesssim  \|f_{ext}\|_{L_p(L_\infty(\mathbb{R})\overline{\otimes}\mathcal{M})}^p.
 \end{align*}
  Now, $\|f\|_{L_p(L_\infty(\mathbb{Z})\overline{\otimes}\mathcal{M})}\leq\sum_{i=1}^4\|f_i\|_{L_p(L_\infty(\mathbb{Z})\overline{\otimes}\mathcal{M})}\lesssim  \|f_{ext}\|_{L_p(L_\infty(\mathbb{R})\overline{\otimes}\mathcal{M})}$,
 this completes the proof for $1\leq p<\infty$. For $p=\infty$, the proof is indeed simpler as we only need to consider the supremum norms directly.
\end{proof}

\begin{lemma}\label{extnorm1}
For $1\leq p\leq\infty$, if $(f_\lambda)_{\lambda>0}\in L_p(L_\infty(\mathbb{Z})\overline{\otimes}\mathcal{M};\ell_\infty)$, then we have
\begin{equation*}
  \Big\|{\sup_{\lambda>0}}^+f_\lambda\Big\|_{L_p(L_\infty(\mathbb{Z})\overline{\otimes}\mathcal{M})}\sim  \Big\|{\sup_{\lambda>0}}^+(f_\lambda)_{ext}\Big\|_{L_p(L_\infty(\mathbb{R})\overline{\otimes}\mathcal{M})}.
\end{equation*}
\end{lemma}
\begin{proof}
Firstly, we assume that $f_\lambda$ is self-adjoint for $\lambda>0$.
Since $f_\lambda$ is self-adjoint for $\lambda>0$, then
for $\epsilon>0$, there exists $0<a_\epsilon\in L_p(L_\infty(\mathbb{Z})\overline{\otimes}\mathcal{M})$ such that $-a_\epsilon\leq f_\lambda\leq a_\epsilon$ for $\lambda>0$, and
 \begin{equation*}
   \|a_\epsilon\|_{L_p(L_\infty(\mathbb{Z})\overline{\otimes}\mathcal{M})}\leq \Big\|{\sup_{\lambda>0}}^+f_\lambda\Big\|_{L_p(L_\infty(\mathbb{Z})\overline{\otimes}\mathcal{M})}+\epsilon.
 \end{equation*}
As in Lemma \ref{extnorm}, we have
 \begin{equation*}
   \sum_{n\in{\mathbb{Z}}}a_\epsilon(n)\Phi(x-n)\leq\Big(\sum_{n\in{\mathbb{Z}}}|a_\epsilon(n)|^p\Phi(x-n)\Big)^{1/p}\Big(\sum_{n\in{\mathbb{Z}}}\Phi(x-n)\Big)^{(p-1)/p}.
 \end{equation*}
 and
 \begin{equation*}
  -\sum_{n\in{\mathbb{Z}}}a_\epsilon(n)\Phi(x-n) \leq (f_\lambda)_{ext}(x)=\sum_{n\in{\mathbb{Z}}}f_\lambda(n)\Phi(x-n)\leq\sum_{n\in{\mathbb{Z}}}a_\epsilon(n)\Phi(x-n).
 \end{equation*}
Hence we obtain
 \begin{align*}
  \Big\|{\sup_{\lambda>0}}^+(f_\lambda)_{ext}\Big\|_{L_p(L_\infty(\mathbb{R})\overline{\otimes}\mathcal{M})}&\leq \Big\| \Big(\sum_{n\in{\mathbb{Z}}}|a_\epsilon(n)|^p\Phi(x-n)\Big)^{1/p}\Big(\sum_{n\in \mathbb{Z}}\Phi(x-n)\Big)^{(p-1)/p}\Big\|_{L_p(L_\infty(\mathbb{R})\overline{\otimes}\mathcal{M})}\\
   &\lesssim \|a_\epsilon\|_{L_p(L_\infty(\mathbb{Z})\overline{\otimes}\mathcal{M})}\leq \Big\|{\sup_{\lambda>0}}^+f_\lambda\Big\|_{L_p(L_\infty(\mathbb{Z})\overline{\otimes}\mathcal{M})}+\epsilon.
\end{align*}
 Let $\epsilon\rightarrow0$, we deduce
 \begin{equation}\label{eqnorms1}
 \Big\|{\sup_{\lambda>0}}^+(f_\lambda)_{ext}\Big\|_{L_p(L_\infty(\mathbb{R})\overline{\otimes}\mathcal{M})}\lesssim\|{\sup_{\lambda>0}}^+f_\lambda\|_{L_p(L_\infty(\mathbb{Z})\overline{\otimes}\mathcal{M})}.
 \end{equation}
 For a general $f_\lambda$, we consider its real and imaginary parts, $\Re(f_\lambda)=\frac{f_\lambda+f^*_\lambda}{2}$ and $ \Im(f_\lambda)=\frac{f_\lambda-f^*_\lambda}{2i}$ respectively, by applying the triangle inequality along with inequality \eqref{eqnorms1}, we obtain the following inequality
 \begin{align*}
   \Big\|{\sup_{\lambda>0}}^+(f_\lambda)_{ext}\Big\|_{L_p(L_\infty(\mathbb{R})\overline{\otimes}\mathcal{M})}&\leq\Big\|{\sup_{\lambda>0}}^+(\Re(f_\lambda))_{ext}\Big\|_{L_p(L_\infty(\mathbb{R})\overline{\otimes}\mathcal{M})}+\Big\|{\sup_{\lambda>0}}^+(\Im(f_\lambda))_{ext}\Big\|_{L_p(L_\infty(\mathbb{R})\overline{\otimes}\mathcal{M})}\\
   &\leq \Big\|{\sup_{\lambda>0}}^+\Re(f_\lambda)\Big\|_{L_p(L_\infty(\mathbb{Z})\overline{\otimes}\mathcal{M})}+\Big\|{\sup_{\lambda>0}}^+\Im(f_\lambda)\Big\|_{L_p(L_\infty(\mathbb{Z})\overline{\otimes}\mathcal{M})}\\
   &\lesssim \Big\|{\sup_{\lambda>0}}^+f_\lambda\Big\|_{L_p(L_\infty(\mathbb{Z})\overline{\otimes}\mathcal{M})}+\Big\|{\sup_{\lambda>0}}^+f_\lambda^*\Big\|_{L_p(L_\infty(\mathbb{Z})\overline{\otimes}\mathcal{M})}\lesssim\Big\|{\sup_{\lambda>0}}^+f_\lambda\Big\|_{L_p(L_\infty(\mathbb{Z})\overline{\otimes}\mathcal{M})}.
 \end{align*}
 Moreover, by combining the arguments presented above with the same technique employed in Lemma \ref{extnorm}, we can derive the reverse inequality as well.
\end{proof}

Utilizing the lemmas previously established, we arrive at the following noncommutative sampling principle.
\begin{lemma}\label{sampling1}
Let $\{\phi_\lambda\}_{\lambda>0}$ be a sequence of smooth functions on $\mathbb{R}$ and $supp(\phi_\lambda)\subset(-\frac{1}{2},\frac{1}{2})$, $T_{\phi_\lambda}$ is the Fourier multiplier operator associated with symbol $\phi_\lambda$.  For $1\leq p\leq \infty$, if the estimate
\begin{equation}\label{sm1}
  \Big\|{\sup_{\lambda>0}}^+T_{\phi_\lambda}(f)\Big\|_{L_p(L_\infty(\mathbb{R})\overline{\otimes}\mathcal{M})}\lesssim\|f\|_{L_p(L_\infty(\mathbb{R})\overline{\otimes}\mathcal{M})}
\end{equation}
holds for all $f\in {L_p(L_\infty(\mathbb{R})\overline{\otimes}\mathcal{M})}$, then the following estimate
\begin{equation}\label{sm2}
  \Big\|{\sup_{\lambda>0}}^+(T_{\phi_{\lambda}})_{dis}(f)\Big\|_{L_p(L_\infty(\mathbb{Z})\overline{\otimes}\mathcal{M})}\lesssim\|f\|_{L_p(L_\infty(\mathbb{Z})\overline{\otimes}\mathcal{M})}
\end{equation}
holds for all $f\in {L_p(L_\infty(\mathbb{Z})\overline{\otimes}\mathcal{M})}$ , where $(T_{\phi_{\lambda}})_{dis}(f)(n):=\sum_{m\in \mathbb{Z}}f(n-m)\big(\phi_{\lambda}\big)^{\vee}(m)$.
\end{lemma}
\begin{proof}

  Define a function $\widetilde{\phi_\lambda}$ on $\mathbb{R}$ by
\begin{equation*}
  \widetilde{\phi_\lambda}(\xi):=\sum_{|\ell|\leq 1}\phi_\lambda(\xi+\ell),
\end{equation*}
and let $T_{\widetilde{\phi_\lambda}}$ be the Fourier multiplier operator associated with $\widetilde{\phi_\lambda}$. By applying the estimate \eqref{sm1} and the triangle inequality, we obtain
\begin{equation}\label{sm3}
  \Big\|{\sup_{\lambda>0}}^+T_{\widetilde{\phi_\lambda}}(f)\Big\|_{L_p(L_\infty(\mathbb{R})\overline{\otimes}\mathcal{M})}\leq\sum_{|\ell|\leq 1}\Big\|{\sup_{\lambda>0}}^+T_{\phi_\lambda}(f)\Big\|_{L_p(L_\infty(\mathbb{R})\overline{\otimes}\mathcal{M})}\lesssim \|f\|_{L_p(L_\infty(\mathbb{R})\overline{\otimes}\mathcal{M})}.
\end{equation}

Furthermore, one can verify that (see \cite[pages 194-195]{MSW})
\begin{equation}\label{sm4}
  T_{\widetilde{\phi_\lambda}}(f_{ext})=\big((T_{\phi_{\lambda}})_{dis}(f)\big)_{ext}.
\end{equation}
Employing Lemmas \ref{extnorm} and \ref{extnorm1} along with the inequality \eqref{sm3} and equality \eqref{sm4}, we deduce
\begin{align*}
 \Big\|{\sup_{\lambda>0}}^+(T_{\phi_{\lambda}})_{dis}(f)\Big\|_{L_p(L_\infty(\mathbb{Z})\overline{\otimes}\mathcal{M})}&\lesssim\Big\|{\sup_{\lambda>0}}^+\big((T_{\phi_{\lambda}})_{dis}(f)\big)_{ext}\Big\|_{L_p(L_\infty(\mathbb{R})\overline{\otimes}\mathcal{M})}\\
 &= \Big\|{\sup_{\lambda>0}}^+T_{\widetilde{\phi_\lambda}}(f_{ext})\Big\|_{L_p(L_\infty(\mathbb{R})\overline{\otimes}\mathcal{M})}\\
 &\lesssim \|f_{ext}\|_{L_p(L_\infty(\mathbb{R})\overline{\otimes}\mathcal{M})}\lesssim \|f\|_{L_p(L_\infty(\mathbb{Z})\overline{\otimes}\mathcal{M})}.
\end{align*}
\end{proof}

To effectively apply the sampling principle in our context, we require a modified version of the noncommutative sampling principle.
Let $q$ be a positive integer, and define $\alpha$ as the isomorphism from $q\mathbb{Z}\times \mathbb{Z}/q\mathbb{Z}$ to $\mathbb{Z}$, where
\begin{equation*}
  \alpha(n_1q,n_2):=n_1q+n_2, \quad n_1\in \mathbb{Z},\ 0\leq n_2<q.
\end{equation*}
Now, given an operator-valued function $f$ mapping from $\mathbb{Z}$ to a semifinite von Neumann algebra $\mathcal{M}$, we introduce the associated function $\alpha(f)$ which maps from
$q\mathbb{Z}$ to $L_\infty(\mathbb{Z}/q\mathbb{Z})\overline{\otimes}\mathcal{M}$ as follows,
\begin{equation*}
 \big((\alpha(f))(n_1q)\big)(n_2):=f(n_1q+n_2), \quad n_1\in \mathbb{Z},\ 0\leq n_2<q.
\end{equation*}
Furthermore, we define a function $\sigma(f)$ mapping from $q\mathbb{Z}$ to $\mathcal{M}$ according to
\begin{equation*}
  (\sigma(f))(nq):=f(n), \quad n\in\mathbb{Z}.
\end{equation*}
Note that
\begin{equation*}
 \tau\Big(\sum_{n\in\mathbb{Z}}\big|\sigma(f)(nq)\big|^p\Big) =\tau\Big(\sum_{n\in\mathbb{Z}}\big|f(n)\big|^p\Big)=\tau\Big(\sum_{n_1\in\mathbb{Z}}\sum_{n_2=0}^{q-1}\big|f(n_1q+n_2)\big|^p\Big),
\end{equation*}
we have
\begin{align}\label{slastt}
  \|\sigma(f)\|_{L_p(L_\infty(q\mathbb{Z})\overline{\otimes}\mathcal{M})}^p=\|f\|_{L_p(L_\infty(\mathbb{Z})\overline{\otimes}\mathcal{M})}^p
 =\|\alpha(f)\|_{L_p\big(L_\infty(q\mathbb{Z})\overline{\otimes}L_\infty(\mathbb{Z}/q\mathbb{Z})\overline{\otimes}\mathcal{M}\big)}^p .
 %&=\|\gamma(f)\|_{L_p\big(L_\infty(\mathbb{Z})\overline{\otimes}L_\infty(\mathbb{Z}/q\mathbb{Z})\overline{\otimes}\mathcal{M}\big)}^p.
\end{align}
Furthermore, we establish the following crucial equivalence of the maximal norms.
\begin{lemma}\label{slast1}
 Let $\{f_\lambda\}_{\lambda>0}\in L_p(L_\infty(\mathbb{Z})\overline{\otimes}\mathcal{M};\ell_\infty)$, then, the following holds,
 \begin{align*}
 \Big\|{\sup_{\lambda>0}}^+f_\lambda\Big\|_{L_{p}(L_\infty(\mathbb{Z})\overline{\otimes}\mathcal{M})}&\sim
  \Big\|{\sup_{\lambda>0}}^+\alpha(f_\lambda)\Big\|_{L_p\big(L_\infty(q\mathbb{Z})\overline{\otimes}L_\infty(\mathbb{Z}/q\mathbb{Z})\overline{\otimes}\mathcal{M}\big)}\\
  &\sim\Big\|{\sup_{\lambda>0}}^+\sigma(f_\lambda)\Big\|_{L_p\big(L_\infty(q\mathbb{Z})\overline{\otimes}\mathcal{M}\big)}.
\end{align*}
\end{lemma}
\begin{proof}Let $\epsilon>0$ be arbitrary. Since $\{f_\lambda\}_{\lambda>0}\in L_p(L_\infty(\mathbb{Z})\overline{\otimes}\mathcal{M};\ell_\infty)$, there exists $a,b,F_\lambda$, such that
$f_\lambda=aF_\lambda b$, and
\begin{equation*}
  \|a\|_{L_{2p}(L_\infty(\mathbb{Z})\overline{\otimes}\mathcal{M})}{\sup_{\lambda>0}}\|F_\lambda\|_{L_\infty(\mathbb{Z})\overline{\otimes}\mathcal{M}}\|b\|_{L_{2p}(L_\infty(\mathbb{Z})\overline{\otimes}\mathcal{M})}
\leq \Big\|{\sup_{\lambda>0}}^+f_\lambda\Big\|_{L_{p}(L_\infty(\mathbb{Z})\overline{\otimes}\mathcal{M})}+\epsilon
\end{equation*}
Hence $\alpha(f_\lambda)=\alpha(a)\alpha(F_\lambda)\alpha(b)$, then we have
\begin{align*}
\Big\|{\sup_{\lambda>0}}^+\alpha(f_\lambda)\Big\|_{L_p\big(L_\infty(q\mathbb{Z})\overline{\otimes}L_\infty(\mathbb{Z}/q\mathbb{Z})\overline{\otimes}\mathcal{M}\big)}&\leq\|\alpha(a)\|_{2p}{\sup_{\lambda>0}}\|\alpha(F_\lambda)\|_\infty\|\alpha(b)\|_{2p}\nonumber\\
&\leq \|a\|_{L_{2p}(L_\infty(\mathbb{Z})\overline{\otimes}\mathcal{M})}{\sup_{\lambda>0}}\|F_\lambda\|_{L_\infty(\mathbb{Z})\overline{\otimes}\mathcal{M}}\|b\|_{L_{2p}(L_\infty(\mathbb{Z})\overline{\otimes}\mathcal{M})}
\nonumber\\
& \leq\Big\|{\sup_{\lambda>0}}^+f_\lambda\Big\|_{L_{p}(L_\infty(\mathbb{Z})\overline{\otimes}\mathcal{M})}+\epsilon.
\end{align*}
Let $\epsilon\rightarrow0$, we have
\begin{equation*}
  \Big\|{\sup_{\lambda>0}}^+\alpha(f_\lambda)\Big\|_{L_p\big(L_\infty(q\mathbb{Z})\overline{\otimes}L_\infty(\mathbb{Z}/q\mathbb{Z})\overline{\otimes}\mathcal{M}\big)}\leq
\Big\|{\sup_{\lambda>0}}^+f_\lambda\Big\|_{L_{p}(L_\infty(\mathbb{Z})\overline{\otimes}\mathcal{M})}.
\end{equation*}
Via a similar argument, we can deduce the inverse inequality
\begin{equation*}
 \Big\|{\sup_{\lambda>0}}^+f_\lambda\Big\|_{L_{p}(L_\infty(\mathbb{Z})\overline{\otimes}\mathcal{M})}\leq
  \Big\|{\sup_{\lambda>0}}^+\alpha(f_\lambda)\Big\|_{L_p\big(L_\infty(q\mathbb{Z})\overline{\otimes}L_\infty(\mathbb{Z}/q\mathbb{Z})\overline{\otimes}\mathcal{M}\big)}.
\end{equation*}
Therefore
\begin{equation*}
 \Big\|{\sup_{\lambda>0}}^+f_\lambda\Big\|_{L_{p}(L_\infty(\mathbb{Z})\overline{\otimes}\mathcal{M})}\sim
  \Big\|{\sup_{\lambda>0}}^+\alpha(f_\lambda)\Big\|_{L_p\big(L_\infty(q\mathbb{Z})\overline{\otimes}L_\infty(\mathbb{Z}/q\mathbb{Z})\overline{\otimes}\mathcal{M}\big)}.
\end{equation*}
Moreover, we also have
\begin{equation*}
 \Big\|{\sup_{\lambda>0}}^+f_\lambda\Big\|_{L_{p}(L_\infty(\mathbb{Z})\overline{\otimes}\mathcal{M})}\sim
  \Big\|{\sup_{\lambda>0}}^+\sigma(f_\lambda)\Big\|_{L_p\big(L_\infty(q\mathbb{Z})\overline{\otimes}\mathcal{M}\big)}.
\end{equation*}
\end{proof}
%\begin{remark}
%  Given a $\mathcal{M}$-valued function $f$ on $q\mathbb{Z}$, we can define a function $\sigma^{-1}(f)$ on $\mathbb{Z}$, such that
%
%\end{remark}

%and
%\begin{equation}\label{slast3}
% \Big\|\sup_{t>0}f_t\Big\|_{L_{p}(L_\infty(\mathbb{Z})\overline{\otimes}\mathcal{M})}\sim
%  \Big\|\sup_{t>0}\gamma(f_t)\Big\|_{L_p\big(L_\infty(\mathbb{Z})\overline{\otimes}L_\infty(\mathbb{Z}/q\mathbb{Z})\overline{\otimes}\mathcal{M}\big)}.
%\end{equation}
%\begin{remark}
%  Indeed, one can see that $\alpha(f)=\sigma(\gamma(f)).$
%\end{remark}
%\begin{align*}
%  \Big\|\sup_{t>0}(T_{\phi_t}^q)_{dis}(f)\Big\|_{L_{p}(L_\infty(\mathbb{Z})\overline{\otimes}\mathcal{M})}\leq
%  \Big\|\sup_{t>0}\alpha((T_{\phi_t}^q)_{dis}(f))\Big\|_{L_p\big(L_\infty(q\mathbb{Z})\overline{\otimes}L_\infty(\mathbb{Z}/q\mathbb{Z})\overline{\otimes}\mathcal{M}\big)}
% \end{align*}
%Given a $\mathcal{M}$- valued function $g$ on $q\mathbb{Z}\times \mathbb{Z}/q\mathbb{Z}$, we define a $L_\infty(\mathbb{Z}/q\mathbb{Z})\overline{\otimes}\mathcal{M}$-valued function $\sigma(g)$ on $q\mathbb{Z}$ by
%\begin{equation*}
% \big(( \sigma(g))(n_1 q)\big)(n_2):=g(n_1q, n_2), \quad n_1\in \mathbb{Z},\ 0\leq n_2<q.
%\end{equation*}
Now, consider a smooth function $\phi$ on $\mathbb{R}$ that is supported on $(-\frac{1}{2q},\frac{1}{2q})$. One can define $\phi_{per}^q$ as
\begin{equation*}
 \phi_{per}^q(\xi):=\sum_{n\in\mathbb{Z}}\phi(\xi-n/q).
\end{equation*}
Let $(T_\phi^q)_{dis}$ be the Fourier multiplier operator associated with symbol $\phi_{per}^q$, which means that for suitable $f$, we have
\begin{equation}\label{Tdef}
 \sum_{n\in\mathbb{Z}}((T_\phi^q)_{dis}(f))(n)e^{-2\pi in\xi}=\phi_{per}^q(\xi)\sum_{n\in\mathbb{Z}}f(n)e^{-2\pi in\xi}.
\end{equation}
To understand the operator better, let us determine its kernel $K_{dis}^q$. We compute
\begin{align*}
  K_{dis}^q(m)&=\int_{(-\frac{1}{2},\frac{1}{2})}\Big(\sum_{n\in\mathbb{Z}}\phi(\xi-n/q)\Big)e^{2\pi im\xi}d\xi\\
  &=\sum_{n\in\mathbb{Z}}\int_{(-\frac{1}{2}-\frac{n}{q},\frac{1}{2}-\frac{n}{q})}\phi(\xi)e^{2\pi im\xi}d\xi e^{2\pi i\frac{mn}{q}}\\
  &=\sum_{n\in\mathbb{Z}}\int_{(-\frac{1}{2}-\frac{n}{q},\frac{1}{2}-\frac{n}{q})\cap(-\frac{1}{2q}.\frac{1}{2q})}\phi(\xi)e^{2\pi im\xi}d\xi e^{2\pi i\frac{mn}{q}}\\
  &=\sum_{n\in\mathbb{Z},|n|<\frac{q+1}{2}}\int_{(-\frac{1}{2}-\frac{n}{q},\frac{1}{2}-\frac{n}{q})\cap(-\frac{1}{2q}.\frac{1}{2q})}\phi(\xi)e^{2\pi im\xi}d\xi e^{2\pi i\frac{mn}{q}}
\end{align*}
Observing the above expression, one can find that $K_{dis}^q(m)=0$ if $m$ is not divisible by $q$, and  $K_{dis}^q(m)=q\check{\phi}(m)$ if $q$ divides $m$.
Hence
\begin{equation}\label{slast3}
  ((T_\phi^q)_{dis}(f))(n)=q\sum_{m\in \mathbb{Z}}f(n-mq)(\phi)^\vee(mq).
\end{equation}

\begin{lemma}\label{sampling2}
Let $\{\phi_\lambda\}_{\lambda>0}$ be a sequence of smooth functions on $\mathbb{R}$ and $supp(\phi_\lambda)\subset(-\frac{1}{2},\frac{1}{2})$, $T_{\phi_\lambda}$ is the Fourier multiplier operator with the multiplier function $\phi_\lambda$. For $1\leq p\leq \infty$, if the estimate
\begin{equation}\label{sm21}
  \Big\|{\sup_{\lambda>0}}^+T_{\phi_\lambda}(g)\Big\|_{L_p\big(L_\infty(\mathbb{R})\overline{\otimes}\mathcal{M}\big)}
  \lesssim\|g\|_{L_p\big(L_\infty(\mathbb{R})\overline{\otimes}\mathcal{M}\big)}
\end{equation}
holds for all $g\in {L_p\big(L_\infty(\mathbb{R})\overline{\otimes}\mathcal{M}\big)}$ and arbitrary semifinite von Neumann algebra $\mathcal{M}$. Then the following estimate
\begin{equation}\label{sm22}
  \Big\|{\sup_{\lambda>0}}^+(T_{\phi_{\lambda,q}}^q)_{dis}(f)\Big\|_{L_p\big(L_\infty(\mathbb{Z})\overline{\otimes}\mathcal{M}\big)}
  \lesssim\|f\|_{L_p\big(L_\infty(\mathbb{Z})\overline{\otimes}\mathcal{M}\big)}
\end{equation}
holds for all $f\in L_p\big(L_\infty(\mathbb{Z})\overline{\otimes}\mathcal{M}\big)$ and arbitrary semifinite von Neumann algebra $\mathcal{M}$, where $\phi_{\lambda,q}(\xi):=\phi_\lambda(q\xi)$.
\end{lemma}

\begin{proof}
Fix a semifinite von Neumann algebra $\mathcal{M}$, given $f\in L_p\big(L_\infty(\mathbb{Z})\overline{\otimes}\mathcal{M}\big)$.
Utilizing equality \eqref{slast3}, we derive the following equations
\begin{align*}
  \big(\alpha((T_{\phi_{\lambda,q}}^q)_{dis}(f))(n_1 q)\big)(n_2)&=\big((T_{\phi_{\lambda,q}}^q)_{dis}(f)\big)(n_1q+n_2)\\
  &=q\sum_{m\in \mathbb{Z}}f(n_1q+n_2-mq)(\phi_{\lambda,q})^\vee(mq)\\
  &=q\sum_{m\in \mathbb{Z}}\big((\alpha(f))((n_1-m) q)\big)(n_2)(\phi_{\lambda,q})^\vee(mq)\\
  &=\sum_{m\in \mathbb{Z}}\big((\sigma^{-1}(\alpha(f)))(n_1-m)\big)(n_2)(\phi_\lambda)^\vee(m).
\end{align*}
This implies that
\begin{equation}\label{slastt1}
  \sigma^{-1}(\alpha((T_{\phi_\lambda}^q)_{dis}(f)))(n_1)=\alpha((T_{\phi_\lambda}^q)_{dis}(f))(n_1q)=(T_{\phi_\lambda})_{dis}(\sigma^{-1}(\alpha(f)))(n_1).
\end{equation}
Here, $\sigma^{-1}$ denotes the map such that $\sigma^{-1}\circ\sigma=id$.
By invoking inequality \eqref{sm21},  Lemma \ref{sampling1} implies the following inequality
\begin{equation}\label{slastt2}
  \Big\|{\sup_{\lambda>0}}^+(T_{\phi_\lambda})_{dis}(g)\Big\|_{L_p\big(L_\infty(\mathbb{Z})\overline{\otimes}L_\infty(\mathbb{Z}/q\mathbb{Z})\overline{\otimes}\mathcal{M}\big)}
  \lesssim\|g\|_{L_p\big(L_\infty(\mathbb{Z})\overline{\otimes}L_\infty(\mathbb{Z}/q\mathbb{Z})\overline{\otimes}\mathcal{M}\big)}
\end{equation}
holds for all $g\in L_\infty(\mathbb{Z})\overline{\otimes}L_\infty(\mathbb{Z}/q\mathbb{Z})\overline{\otimes}\mathcal{M}$.
Finally, by combining the inequalities \eqref{slastt}, \eqref{slastt1}, \eqref{slastt2}, along with Lemma \ref{slast1}, we obtain
\begin{align*}
  \Big\|{\sup_{\lambda>0}}^+(T_{\phi_{\lambda,q}}^q)_{dis}(f)\Big\|_{L_p\big(L_\infty(\mathbb{Z})\overline{\otimes}\mathcal{M}\big)}
  &\sim\Big\|{\sup_{\lambda>0}}^+\big(\alpha((T_{\phi_{\lambda,q}}^q)_{dis}(f))\Big\|_{L_p\big(L_\infty(q\mathbb{Z})\overline{\otimes}L_\infty(\mathbb{Z}/q\mathbb{Z})\overline{\otimes}\mathcal{M}\big)}\\
   &\sim\Big\|{\sup_{\lambda>0}}^+\sigma^{-1}(\alpha((T_{\phi_\lambda}^q)_{dis}(f)))\Big\|_{L_p\big(L_\infty(\mathbb{Z})\overline{\otimes}L_\infty(\mathbb{Z}/q\mathbb{Z})\overline{\otimes}\mathcal{M}\big)}\\
  &=\Big\|{\sup_{\lambda>0}}^+(T_{\phi_\lambda})_{dis}(\sigma^{-1}(\alpha(f)))\Big\|_{L_p\big(L_\infty(\mathbb{Z})\overline{\otimes}L_\infty(\mathbb{Z}/q\mathbb{Z})\overline{\otimes}\mathcal{M}\big)}\\
  &\lesssim\Big\|\sigma^{-1}(\alpha(f))\Big\|_{L_p\big(L_\infty(\mathbb{Z})\overline{\otimes}L_\infty(\mathbb{Z}/q\mathbb{Z})\overline{\otimes}\mathcal{M}\big)}\\
  &\lesssim\Big\|\alpha(f)\Big\|_{L_p\big(L_\infty(q\mathbb{Z})\overline{\otimes}L_\infty(\mathbb{Z}/q\mathbb{Z})\overline{\otimes}\mathcal{M}\big)}\\
  &\lesssim\Big\|f\Big\|_{L_p\big(L_\infty(\mathbb{Z})\overline{\otimes}\mathcal{M}\big)}.
\end{align*}
\end{proof}

\section{Proof of the main result}\label{last}
In this section, $\mathcal{N}$ denotes a semifinite von Neumann algebra, fixed as per the notation employed in Theorem \ref{thm1}.

\subsection{Some lemmas}
The lemmas presented in this subsection are crucial to our ultimate goal of deriving the noncommutative maximal ergodic inequality. These lemmas provide key estimates and properties that will be used in subsequent derivations.
\begin{lemma}\label{kt}
Let $k(x):=t^{-1}x^{1/t-1}\chi_{[0,1]}(x)$, where $t>0$. For $1<p<\infty$, we have the following inequality,
\begin{equation*}
  \Big\|{\sup_{\eta>0}}^+f*k_\eta\Big\|_{L_p(L_\infty(\mathbb{R})\overline{\otimes}{\mathcal{M}})}\leq C\|f\|_{L_p(L_\infty(\mathbb{R})\overline{\otimes}{\mathcal{M}})},
\end{equation*}
where $k_\eta(x)=\frac{1}{\eta}k\big(\frac{x}{\eta}\big)$ and ${\mathcal{M}}$ is an arbitrary semifinite von Neumann algebra. The constant $C$ only depends on $p$ and $t$.
\end{lemma}
\begin{proof}
  Without loss of generality, we assume that $f\in \mathcal{S}(\mathbb{R})_+\otimes\mathcal{S}({\mathcal{M}})_+$. Then
  \begin{align*}
    (f*k_\eta)(x)&=\frac{1}{\eta}\int_0^\eta t^{-1}f(x-y)(\frac{y}{\eta})^{1/t-1}dy\\
    &=\sum_{k=0}^\infty\frac{1}{\eta}\int_{2^{-(k+1)}\eta}^{2^{-k}\eta}t^{-1}f(x-y)(\frac{y}{\eta})^{1/t-1}dy\\
    %&=\sum_{k=0}^\infty\frac{1}{\eta}\int_{2^{-(k+1)}\eta}^{2^{-k}\eta}t^{-1}f(x-y)(\frac{y}{\eta})^{1/t-1}dy\\
    %&\lesssim\sum_{k=0}^\infty\frac{1}{\eta}\int_{2^{-k}\eta}^{2^{-k}\eta}t^{-1}f(x-y)(2^{-k})^{1/t-1}dy\\
    &\lesssim\sum_{k=0}^\infty\frac{(2^{-k})^{1/t}}{2^{-k}\eta}\int_{2^{-(k+1)}\eta}^{2^{-k}\eta}t^{-1}f(x-y)dy.
  \end{align*}
Hence by the triangle inequality and the noncommutative Hardy-Littlewood maximal inequality (see e.g. \cite[Theorem 3.3]{Mei09}), we obtain
\begin{align*}
  \Big\|{\sup_{\eta>0}}^+f*k_\eta\Big\|_{L_p(L_\infty(\mathbb{R})\overline{\otimes}{\mathcal{M}})}&\lesssim \sum_{k=0}^\infty t^{-1}2^{-k/t}\Big\|{\sup_{\eta>0}}^+\frac{1}{2^{-k}\eta}\int_{2^{-(k+1)}\eta}^{2^{-k}\eta}f(x-y)dy\Big\|_{L_p(L_\infty(\mathbb{R})\overline{\otimes}{\mathcal{M}})}\\
  &\lesssim\sum_{k=0}^\infty t^{-1}2^{-k/t}\|f\|_{L_p(L_\infty(\mathbb{R})\overline{\otimes}{\mathcal{M}})}\lesssim\|f\|_{L_p(L_\infty(\mathbb{R})\overline{\otimes}{\mathcal{M}})}.
\end{align*}
\end{proof}
\begin{lemma}\label{sampling}
Let $\varphi$ be a smooth function supported on $(-\frac{1}{2},\frac{1}{2})$ in $\mathbb{R}$ and let $1\leq q<D$. Then
 \begin{equation}\label{sampes1}
     \Big\|{\sup_{\lambda>0}}^+\sum_{0\leq a<q}\int_{\mathbb{T}} \hat{k}(\lambda\beta)\varphi(D\beta)\hat{f}(a/q+\beta)e^{2\pi im(a/q+\beta)}d\beta\Big\|_{L_p(L_\infty(\mathbb{Z})\overline{\otimes}\mathcal{N})}\leq C\|f\|_{L_p(L_\infty(\mathbb{Z})\overline{\otimes}\mathcal{N})}.
     \end{equation}
\end{lemma}
\begin{proof}
  Let $\tilde{\varphi}_\lambda(\xi):=\hat{k}(\frac{\lambda \xi}{q})\varphi(\frac{D\xi}{q})$, then $supp(\tilde{\varphi}_\lambda)\subset(-\frac{1}{2},\frac{1}{2})$ and we have
  \begin{align}\label{sampes2}
    \sum_{0\leq a<q}\int_{\mathbb{T}} \hat{k}(\lambda(\beta-a/q))\varphi(D(\beta-a/q))\hat{f}(\beta)e^{2\pi im\beta}d\beta=\big((T^q_{\tilde{\varphi}_{\lambda,q}})_{dis}(f)\big)(m),
  \end{align}
  where $(T^q_{\tilde{\varphi}_{\lambda,q}})_{dis}$ is defined as in equality \eqref{Tdef}.
  By applying Lemma \ref{kt} and the Young inequality, for all $g\in {L_p\big(L_\infty(\mathbb{R})\overline{\otimes}\mathcal{M}\big)}$ and arbitrary semifinite von Neumann algebra $\mathcal{M}$, we have
  \begin{align}\label{sampes3}
   \Big\|{\sup_{\lambda>0}}^+T_{\tilde{\varphi}_\lambda}(g)\Big\|_{L_p\big(L_\infty(\mathbb{R})\overline{\otimes}\mathcal{M}\big)}
   &=\Big\|{\sup_{\lambda>0}}^+g*k_{\lambda/q}*(\check{\varphi})_{D/q}\Big\|_{L_p\big(L_\infty(\mathbb{R})\overline{\otimes}\mathcal{M}\big)}\nonumber\\
&\lesssim\|g*(\check{\varphi})_{D/q}\|_{L_p\big(L_\infty(\mathbb{R})\overline{\otimes}\mathcal{M}\big)}\nonumber\\
&\lesssim\|g\|_{L_p\big(L_\infty(\mathbb{R})\overline{\otimes}\mathcal{M}\big)}.
  \end{align}
  Then combining equality \eqref{sampes2} with Lemma \ref{sampling2}, we obtain
  \begin{align*}
  & \Big\|{\sup_{\lambda>0}}^+\sum_{0\leq a<q}\int_{\mathbb{T}} \hat{k}(\lambda\beta)\varphi(D\beta)\hat{f}(a/q+\beta)e^{2\pi im(a/q+\beta)}d\beta\Big\|_{L_p(L_\infty(\mathbb{Z})\overline{\otimes}\mathcal{N})}\\
  &=\Big\|{\sup_{\lambda>0}}^+\sum_{0\leq a<q}\int_{\mathbb{T}} \hat{k}(\lambda(\beta-a/q))\varphi(D(\beta-a/q))\hat{f}(\beta)e^{2\pi im\beta}d\beta\Big\|_{L_p(L_\infty(\mathbb{Z})\overline{\otimes}\mathcal{N})}\\
  &=\Big\|{\sup_{\lambda>0}}^+(T^q_{\tilde{\varphi}_{\lambda,q}})_{dis}(f)\Big\|_{L_p(L_\infty(\mathbb{Z})\overline{\otimes}\mathcal{N})}\lesssim \|f\|_{L_p(L_\infty(\mathbb{Z})\overline{\otimes}\mathcal{N})}.
  \end{align*}
\end{proof}
\begin{lemma}\label{key1}
  For $p\in (\frac{3\rho+2}{2\rho+1},2]$, we have the following estimate,
  \begin{equation*}
    \Big(\sum_{k=0}^\infty\|f*(K_{2^k}-L_{2^k})\|_{L_p(L_\infty(\mathbb{Z})\overline{\otimes}\mathcal{N})}^p\Big)^{1/p}\leq C\|f\|_{L_p(L_\infty(\mathbb{Z})\overline{\otimes}\mathcal{N})}.
  \end{equation*}
 \end{lemma}
 \begin{proof}
   It suffices to show for any $N>0$, there exists $c(p)>\frac{1}{p}$, such that
   \begin{equation}\label{decayes1}
    \Big\|f*(K_{N}-L_{N})\Big\|_{L_p(L_\infty(\mathbb{Z})\overline{\otimes}\mathcal{N})}\leq C(\log(N))^{-c(p)}\|f\|_{L_p(L_\infty(\mathbb{Z})\overline{\otimes}\mathcal{N})}.
   \end{equation}
   On one hand, by applying Lemma \ref{akey2} and the Young inequality, we have
   \begin{align}\label{decayes2}
     &\Big\|\sum_{\xi\in \mathfrak{R}_s\backslash\mathfrak{R}_{s-1}}\int_{\mathbb{T}} S(\xi)\hat{k}(N^t(\theta-\xi))\phi(D_s(\theta-\xi))\hat{f}(\theta)e^{2\pi i\theta m}d\theta\Big\|_{L_1(L_\infty(\mathbb{Z})\overline{\otimes}\mathcal{N})}\nonumber\\
      &\leq\Big\|\sum_{\xi\in \mathfrak{R}_s}\int_{\mathbb{T}} S(\xi)\hat{k}(N^t(\theta-\xi))\phi(D_s(\theta-\xi))\hat{f}(\theta)e^{2\pi i\theta m}d\theta\Big\|_{L_1(L_\infty(\mathbb{Z})\overline{\otimes}\mathcal{N})}\nonumber\\
      &+\Big\|\sum_{\xi\in \mathfrak{R}_{s-1}}\int_{\mathbb{T}} S(\xi)\hat{k}(N^t(\theta-\xi))\phi(D_s(\theta-\xi))\hat{f}(\theta)e^{2\pi i\theta m}d\theta\Big\|_{L_1(L_\infty(\mathbb{Z})\overline{\otimes}\mathcal{N})}\nonumber\\
      &\leq C2^{s\rho}\|f\|_{L_1(L_\infty(\mathbb{Z})\overline{\otimes}\mathcal{N})}.
   \end{align}
   On the other hand,we observe that for any distinct $\xi$ and  $\eta$ in $\mathfrak{R}_s\setminus\mathfrak{R}_{s-1}$, we have $D_s|\xi-\eta|>2$. This implies that for a fixed $\theta\in \mathbb{T}$,  there can be at most one $\xi\in \mathfrak{R}_s\setminus\mathfrak{R}_{s-1}$ such that $\phi(D_s(\theta-\xi))\neq0$. Utilizing Lemma \ref{akeq4}, one has
    \begin{equation}\label{eqsuu}
      \Big|\sum_{\xi\in \mathfrak{R}_s\backslash\mathfrak{R}_{s-1}} S(\xi)\hat{k}(N^t(\theta-\xi))\phi(D_s(\theta-\xi))\Big|\leq C2^{-\frac{(1+\rho)s}{2}}.
    \end{equation}
    Then by combining the Parseval identity with inequality \eqref{eqsuu}, we obtain
   \begin{align}\label{decayes3}
     &\Big\|\sum_{\xi\in \mathfrak{R}_s\backslash\mathfrak{R}_{s-1}}\int_{\mathbb{T}} S(\xi)\hat{k}(N^t(\theta-\xi))\phi(D_s(\theta-\xi))\hat{f}(\theta)e^{2\pi i\theta m}d\theta\Big\|_{L_2(L_\infty(\mathbb{Z})\overline{\otimes}\mathcal{N})}\nonumber\\
     &=\Big\{\tau\Big(\int_{\mathbb{T}}\Big|\sum_{\xi\in \mathfrak{R}_s\backslash\mathfrak{R}_{s-1}} S(\xi)\hat{k}(N^t(\theta-\xi))\phi(D_s(\theta-\xi))\Big|^2|\hat{f}(\theta)|^2d\theta\Big)\Big\}^{1/2}\nonumber\\
     &\leq C2^{-\frac{(1+\rho)s}{2}} \Big\{\tau\Big(\int_{\mathbb{T}}|\hat{f}(\theta)|^2d\theta\Big)\Big\}^{1/2}= C2^{-\frac{(1+\rho)s}{2}}\|f\|_{L_2(L_\infty(\mathbb{Z})\overline{\otimes}\mathcal{N})}.
   \end{align}
   By interpolation between \eqref{decayes2} and \eqref{decayes3}, for $r\in(1,2)$, we have
     \begin{align}\label{decayes4}
    &\Big\|\sum_{\xi\in \mathfrak{R}_s\backslash\mathfrak{R}_{s-1}}\int_{\mathbb{T}} S(\xi)\hat{k}(N^t(\theta-\xi))\phi(D_s(\theta-\xi))\hat{f}(\theta)e^{2\pi i\theta m}d\theta\Big\|_{L_r(L_\infty(\mathbb{Z})\overline{\otimes}\mathcal{N})}\nonumber\\&\leq C2^{-s(2\rho+1-\frac{3\rho+1}{r})}\|f\|_{L_r(L_\infty(\mathbb{Z})\overline{\otimes}\mathcal{N})}.
     \end{align}
     Hence if $r>\frac{3\rho+1}{2\rho+1}$, by combining the Young inequality with \eqref{decayes4}, we conclude that
 %\begin{equation}\label{decayes5}
%     \Big\|\sum_{s=1}^\infty\sum_{\xi\in \mathfrak{R}_s\backslash\mathfrak{R}_{s-1}}\int_{\mathbb{T}} S(\xi)\hat{k}(N^t(\theta-\xi))\phi(D_s(\theta-\xi))\hat{f}(m)e^{2\pi i\theta m}d\theta\Big\|_{L_r(L_\infty(\mathbb{Z})\overline{\otimes}\mathcal{N})}\leq C\|f\|_{L_r(L_\infty(\mathbb{Z})\overline{\otimes}\mathcal{N})}.
%     \end{equation}
 \begin{align}\label{decayes6}
   \Big\|f*L_{N}\Big\|_{L_r(L_\infty(\mathbb{Z})\overline{\otimes}\mathcal{N})}&\leq \Big\|\int_{\mathbb{T}} \hat{k}(N^t(\theta))\phi(\theta)\hat{f}(\theta)e^{2\pi i\theta m}d\theta\Big\|_{L_r(L_\infty(\mathbb{Z})\overline{\otimes}\mathcal{N})}\nonumber\\&+ \sum_{s=1}^\infty\Big\|\sum_{\xi\in \mathfrak{R}_s\backslash\mathfrak{R}_{s-1}}\int_{\mathbb{T}} S(\xi)\hat{k}(N^t(\theta-\xi))\phi(D_s(\theta-\xi))\hat{f}(\theta)e^{2\pi i\theta m}d\theta\Big\|_{L_r(L_\infty(\mathbb{Z})\overline{\otimes}\mathcal{N})}\nonumber\\&\lesssim \|f\|_{L_r(L_\infty(\mathbb{Z})\overline{\otimes}\mathcal{N})}+\sum_{s=1}^\infty2^{-s(2\rho+1-\frac{3\rho+1}{r})}\|f\|_{L_r(L_\infty(\mathbb{Z})\overline{\otimes}\mathcal{N})}
   \lesssim \|f\|_{L_r(L_\infty(\mathbb{Z})\overline{\otimes}\mathcal{N})}.
   \end{align}
Since it is straightforward to show that  $\Big\|f*K_{N}\Big\|_{L_r(L_\infty(\mathbb{Z})\overline{\otimes}\mathcal{N})}\leq C\|f\|_{L_r(L_\infty(\mathbb{Z})\overline{\otimes}\mathcal{N})}$, we can derive the following inequality,
 \begin{equation}\label{decayes7}
   \Big\|f*(K_N-L_{N})\Big\|_{L_r(L_\infty(\mathbb{Z})\overline{\otimes}\mathcal{N})}\leq C\|f\|_{L_r(L_\infty(\mathbb{Z})\overline{\otimes}\mathcal{N})}.
 \end{equation}
 Applying Lemma \ref{akey3} and the Parseval identity, for $0<\rho^\prime<\rho$, we have
 \begin{equation}\label{decayes8}
   \Big\|f*(K_N-L_{N})\Big\|_{L_2(L_\infty(\mathbb{Z})\overline{\otimes}\mathcal{N})}\leq C(\log(N))^{-\frac{1+\rho^\prime}{2}}\|f\|_{L_2(L_\infty(\mathbb{Z})\overline{\otimes}\mathcal{N})}.
 \end{equation}
 By interpolation again,  fix $p>\frac{3\rho+2}{2\rho+1}$, we arrive at
  \begin{equation}\label{decayes8}
   \Big\|f*(K_N-L_{N})\Big\|_{L_p(L_\infty(\mathbb{Z})\overline{\otimes}\mathcal{N})}\leq C(\log(N))^{-(1+\rho^\prime)\frac{\frac{1}{r}-\frac{1}{p}}{\frac{2}{r}-1}}\|f\|_{L_p(L_\infty(\mathbb{Z})\overline{\otimes}\mathcal{N})}.
 \end{equation}
Then we can choose $\rho^\prime<\rho, r>\frac{3\rho+1}{2\rho+1}$ such that $(1+\rho^\prime)\frac{\frac{1}{r}-\frac{1}{p}}{\frac{2}{r}-1}>\frac{1}{p}$,  thereby completing the proof of this lemma.
 \end{proof}

 \begin{lemma}\label{key2}
  For $p\in(\frac{2\rho+1}{\rho+1},2]$, we have the following estimate,
  \begin{equation*}
   \Big\|{\sup_{N}}^+f*L_{N}\Big\|_{L_p(L_\infty(\mathbb{Z})\overline{\otimes}\mathcal{N})}\leq C\|f\|_{L_p(L_\infty(\mathbb{Z})\overline{\otimes}\mathcal{N})}.
  \end{equation*}
 \end{lemma}

 \begin{proof}
    For a fixed value of $s$, $\Lambda_{p,s}$ denotes the best constant that satisfies the following inequality,
  \begin{equation}\label{last1}
   \Big\|{\sup_{\lambda>0}}^+\sum_{\xi\in\mathfrak{R}_s\setminus\mathfrak{R}_{s-1}}\int_{\mathbb{T}} S(\xi) \hat{f}(\xi+\beta) \hat{k}(\lambda\beta)\phi(D_s\beta)e^{2\pi im(\xi+\beta)}d\beta\Big\|_{L_p(L_\infty(\mathbb{Z})\overline{\otimes}\mathcal{N})}\leq \Lambda_{p,s}\|f\|_{L_p(L_\infty(\mathbb{Z})\overline{\otimes}\mathcal{N})}.
    \end{equation}
  Utilizing this inequality in conjunction with Lemma \ref{kt}, we can draw parallels to the proof of inequality \eqref{decayes6} and arrive at the following deduction,
    \begin{align}\label{last2}
       \Big\|{\sup_{N}}^+f*L_{N}\Big\|_{L_p(L_\infty(\mathbb{Z})\overline{\otimes}\mathcal{N})}
       \lesssim\Big(\sum_{s=1}^\infty\Lambda_{p,s}\Big)\|f\|_{L_p(L_\infty(\mathbb{Z})\overline{\otimes}\mathcal{N})}.
    \end{align}
    Hence it suffices to show that the series $\sum_{s=1}^\infty\Lambda_{p,s}<\infty$. Observe that the set $\mathfrak{R}_s\setminus\mathfrak{R}_{s-1}$ can be expressed as the disjoint union of at most  $2^{s\rho}$ subset $\Gamma_i$ of $\mathbb{T}$, where each $\Gamma_i$ is a difference of cyclic subgroups. Let $\psi$ be a smooth function on $\mathbb{R}$ that equals $1$ on $[-\frac{1}{2},\frac{1}{2}]$ and vanishes outside $[-1,1]$.
    One can define the function
    \begin{equation*}
      F_i(m):=\sum_{\xi\in\Gamma_i}\int_{\mathbb{T}} S(\xi)\psi(D_s(\alpha-\xi))\hat{f}(\alpha)e^{2\pi im\alpha}d\alpha.
    \end{equation*}
   Observing that for distinct $\xi$ and  $\eta$ in $\mathfrak{R}_s\setminus\mathfrak{R}_{s-1}$, we have $D_s|\xi-\eta|>2$. As a result, the support properties of $\phi$ and $\psi$ ensure that $\psi(D_s(\xi+\beta-\eta))\phi(D_s\beta)$ vanishes for all $\beta\in\mathbb{T}$ when $\xi\neq\eta$. Therefore, we can rewrite the expression as follows,
    \begin{align}\label{last3}
  &\sum_{\xi\in\Gamma_i}\int_{\mathbb{T}} S(\xi)\hat{f}(\xi+\beta)\hat{k}(\lambda \beta)\phi(D_s\beta)e^{2\pi im(\xi+\beta)}d\beta\nonumber\\
   &=\sum_{\xi\in\Gamma_i}\int_{\mathbb{T}} \sum_{\eta\in\Gamma_i}\hat{f}(\xi+\beta)S(\eta)\hat{k}(\lambda \beta)\psi(D_s(\xi+\beta-\eta))\phi(D_s\beta)e^{2\pi im(\xi+\beta)}d\beta\nonumber\\
   &=\sum_{\xi\in\Gamma_i}\int_{\mathbb{T}} \hat{F_i}(\xi+\beta)\hat{k}(\lambda \beta)\phi(D_s\beta)e^{2\pi im(\xi+\beta)}d\beta.
    \end{align}
For fixed $\Gamma_i$, we may assume that $\Gamma_i=\mathbb{Z}_{q_1}\backslash\mathbb{Z}_{q_2}$ for some $q_2<q_1<D_s$. Utilizing Lemma \ref{sampling} for $1<r\leq2$ and the triangle inequality, we can derive the following inequality,
     \begin{align}\label{last4}
 &\Big\|{\sup_{\lambda>0}}^+\sum_{\xi\in\Gamma_i}\int_{\mathbb{T}} \hat{F_i}(\xi+\beta)\hat{k}(\lambda \beta)\phi(D_s\beta)e^{2\pi im(\xi+\beta)}d\beta\Big\|_{L_r(L_\infty(\mathbb{Z})\overline{\otimes}\mathcal{N})}\nonumber\\
 &\leq\sum_{\ell=1}^2\Big\|{\sup_{\lambda>0}}^+\sum_{\xi\in\mathbb{Z}_{q_\ell}}\int_{\mathbb{T}} \hat{F_i}(\xi+\beta)\hat{k}(\lambda \beta)\phi(D_s\beta)e^{2\pi im(\xi+\beta)}d\beta\Big\|_{L_r(L_\infty(\mathbb{Z})\overline{\otimes}\mathcal{N})}\nonumber\\
 &\leq C_r\|F_i\|_{L_r(L_\infty(\mathbb{Z})\overline{\otimes}\mathcal{N})}= C_r\Big\|\sum_{\xi\in\Gamma_i}\int_{\mathbb{T}} S(\xi)\psi(D_s(\alpha-\xi))\hat{f}(\alpha)e^{2\pi im\alpha}d\alpha\Big\|_{L_r(L_\infty(\mathbb{Z})\overline{\otimes}\mathcal{N})}.
    \end{align}
  Furthermore, as demonstrated in \cite[page 84]{Bo882}, the inequality
\begin{equation*}
\Big\|\Big(\sum_{\xi\in\Gamma_i}S(\xi)\psi(D_s(\alpha-\xi))\Big)^{\vee}\Big\|_{\ell_1(\mathbb{Z})} < C
\end{equation*}
holds true for a constant C that is independent of $i$. By the equality \eqref{last3}, the triangle inequality, inequality \eqref{last4} and the Young inequality, we get
  \begin{align}\label{last5}
  &\Big\|{\sup_{\lambda>0}}^+\sum_{\xi\in\mathfrak{R}_s\setminus\mathfrak{R}_{s-1}}\int_{\mathbb{T}} S(\xi) \hat{f}(\xi+\beta) \hat{k}(\lambda\beta)\phi(D_s\beta)e^{2\pi im(\xi+\beta)}d\beta\Big\|_{L_r(L_\infty(\mathbb{Z})\overline{\otimes}\mathcal{N})}\nonumber\\
      &=\Big\|{\sup_{\lambda>0}}^+\sum_{\xi\in\mathfrak{R}_s\setminus\mathfrak{R}_{s-1}}\int_{\mathbb{T}} \hat{F_i}(\xi+\beta)\hat{k}(\lambda \beta)\phi(D_s\beta)e^{2\pi im(\xi+\beta)}d\beta\Big\|_{L_r(L_\infty(\mathbb{Z})\overline{\otimes}\mathcal{N})}\nonumber\\
       &\leq\sum_i\Big\|{\sup_{\lambda>0}}^+\sum_{\xi\in\Gamma_i}\int_{\mathbb{T}} \hat{F_i}(\xi+\beta)\hat{k}(\lambda \beta)\phi(D_s\beta)e^{2\pi im(\xi+\beta)}d\beta\Big\|_{L_r(L_\infty(\mathbb{Z})\overline{\otimes}\mathcal{N})}\nonumber\\&\leq C_r\sum_i\Big\|\sum_{\xi\in\Gamma_i}\int_{\mathbb{T}} S(\xi)\psi(D_s(\alpha-\xi))\hat{f}(\alpha)e^{2\pi im\alpha}d\alpha\Big\|_{L_r(L_\infty(\mathbb{Z})\overline{\otimes}\mathcal{N})}\nonumber\\& \leq C_r2^{s\rho}\|f\|_{L_r(L_\infty(\mathbb{Z})\overline{\otimes}\mathcal{N})}.
    \end{align}
    When $r=2$, we can employ the triangle inequality, inequality \eqref{last4}, the Parseval identity, Cauchy-Schwartz inequality and Lemma \ref{akeq4} to derive the following,
\begin{align}\label{last6}
 &\Big\|{\sup_{\lambda>0}}^+\sum_{\xi\in\mathfrak{R}_s\setminus\mathfrak{R}_{s-1}}\int_{\mathbb{T}} \hat{F_i}(\xi+\beta)\hat{k}(\lambda \beta)\phi(D_s\beta)e^{2\pi im(\xi+\beta)}d\beta\Big\|_{L_2(L_\infty(\mathbb{Z})\overline{\otimes}\mathcal{N})}\nonumber\\
 &\leq C\sum_i\Big\|\sum_{\xi\in\Gamma_i}\int_{\mathbb{T}} S(\xi)\psi(D_s(\alpha-\xi))\hat{f}(\alpha)e^{2\pi im\alpha}d\alpha\Big\|_{L_2(L_\infty(\mathbb{Z})\overline{\otimes}\mathcal{N})}\nonumber\\
 & \leq C \sum_i\Big\{\tau\Big(\int_{\mathbb{T}} |\sum_{\xi\in\Gamma_i}S(\xi)\psi(D_s(\alpha-\xi))|^2|\hat{f}(\alpha)|^2d\alpha\Big)\Big\}^{1/2}\nonumber\\
 &=C \sum_i\Big\{\tau\Big(\int_{\mathbb{T}}\sum_{\xi\in\Gamma_i}|S(\xi)\psi(D_s(\alpha-\xi))|^2|\hat{f}(\alpha)|^2d\alpha\Big)\Big\}^{1/2}\nonumber\\
 &\leq C2^{\frac{s\rho}{2}}\Big\{\tau\Big(\int_{\mathbb{T}}\sum_i\sum_{\xi\in\Gamma_i}|S(\xi)\psi(D_s(\alpha-\xi))|^2|\hat{f}(\alpha)|^2d\alpha\Big)\Big\}^{1/2}\nonumber\\
 &\leq C2^{\frac{s\rho}{2}}\cdot 2^{\frac{-s(1+\rho)}{2}}\Big\{\tau\Big(\int_{\mathbb{T}}|\hat{f}(\alpha)|^2d\alpha\Big)\Big\}^{1/2}=C2^{\frac{-s}{2}}\|f\|_{L_2(L_\infty(\mathbb{Z})\overline{\otimes}\mathcal{N})}.
 \end{align}
 The final inequality use the fact that $\psi(D_s(\cdot-\xi))$ for $\xi\in \mathfrak{R}_s\setminus\mathfrak{R}_{s-1}$ have disjoint supports. Now, for $1<r<2$ and $p\in (r,2)$, let $\frac{1}{p}=\frac{1-\theta}{2}+\frac{\theta}{r}$ for some $\theta\in(0,1)$. Interpolating between inequalities \eqref{last5} and \eqref{last6}, we obtain
 \begin{equation*}
   \Lambda_{p,s}\leq C_r2^{\frac{-s(1-\theta)}{2}+s\rho \theta}=C_r2^{\frac{-s(1-(2\rho+1)\theta)}{2}}=C_r2^{\frac{-s\big(1-(\frac{2}{p}-1)\frac{r(2\rho+1)}{2-r}\big)}{2}}.
 \end{equation*}
  If $p>\frac{2\rho+1}{\rho+1}$, we can find $r\in(1,2)$ such that
 \begin{equation*}
   p>\frac{2}{1+\frac{2-r}{r(2\rho+1)}},
 \end{equation*}
 which ensures that $\sum_{s=1}^\infty\Lambda_{p,s}<\infty.$
 \end{proof}
 \subsection{Proof of Theorem \ref{thm1}}
 As in the commutative setting \cite{Bo881, Bo882, Bo89}, Theorem \ref{thm1} can be simplified to the case of the shift operation on $\mathbb{Z}$. More precisely,
 for any $\mathcal{N}$-valued function $f$ defined on $\mathbb{Z}$ with finite support, we introduce the notation
\begin{equation}\label{ea2}
 \mathcal{A}_Nf:=f*K_N, \quad K_N=\frac{1}{N}\sum_{k=1}^N\delta_{\{k^t\}}.
\end{equation}
We claim that it suffices to show for $p>\frac{\sqrt{5}+1}{2}$, the following maximal inequality
\begin{equation}\label{thm11}
   \Big\|{\sup_{N\in\mathbb{N}}}^+\mathcal{A}_Nf\Big\|_{L_p(L_\infty(\mathbb{Z})\overline{\otimes}\mathcal{N})}\leq C_p\|f\|_{L_p(L_\infty(\mathbb{Z})\overline{\otimes}\mathcal{N})}
\end{equation}
holds for all $\mathcal S(\mathcal{N})$-valued function $f$ defined on $\mathbb{Z}$ with finite support.

 Now, let us demonstrate how inequality \eqref{thm11} can be leveraged to establish Theorem \ref{thm1}. Fix $p>\frac{\sqrt{5}+1}{2}$, and without loss of generality, assume that $x\in\mathcal{S}(\mathcal{N})_+$. We then have
 \begin{equation*}
   \Big\|{\sup_{N}}^+A_N(x)\Big\|_{L_p(\mathcal{N})}=\sup_{N}\Big\|{\sup_{1\leq k\leq N}}^+A_k(x)\Big\|_{L_p(\mathcal{N})}.
 \end{equation*}
To prove Theorem \ref{thm1}, it suffices to establish the inequality
\begin{equation}\label{shiftlast}
 \Big\|{\sup_{1\leq k\leq N}}^+A_k(x)\Big\|_{L_p(\mathcal{N})}\leq C_p\|x\|_{L_p(\mathcal{N})}.
\end{equation}
holds for any positive integer $N$.
Fix a positive integer $N$ and define a $\mathcal S(\mathcal{N})$-valued function on $\mathbb{Z}$ as
\begin{equation*}
  \phi(j):=\gamma^{j}(x), \qquad \text{for} \ j\in \mathbb{Z}.
\end{equation*}
Then one can see
\begin{equation*}
  (\mathcal{A}_N\phi)(j)=\frac{1}{N}\sum_{k=1}^N\phi(j+k^t)=\frac{1}{N}\sum_{k=1}^N\gamma^{j+k^t}(x)=A_N(\gamma^j(x)).
\end{equation*}
For a large positive integer $J$, we set $\psi_J(j):=\phi(j)\chi_{1\leq j\leq J}$, then for all $j\leq J-N^t$, we have
\begin{equation}\label{trans4}
  (\mathcal{A}_{N}\psi_J)(j)=A_{N}(\gamma^j(x)).
\end{equation}
Applying inequality \eqref{thm11} to $\psi_J$, we obtain
\begin{align}\label{trans1}
 \Big\|{\sup_{1\leq k\leq N}}^+\mathcal{A}_{k}\psi_J\Big\|_{L_p(L_\infty(\mathbb{Z})\overline{\otimes}\mathcal{N})}^p\lesssim\|\psi_J\|_{L_p(L_\infty(\mathbb{Z})\overline{\otimes}\mathcal{N})}^p
 =\sum_{j\in\mathbb{Z}}\|\psi_J(j)\|_{L_p(\mathcal{N})}^p
 =J\|x\|_{L_p(\mathcal{N})}^p,
\end{align}
where the final equality follows from the fact that $\gamma$ is an isometry in $L_p(\mathcal{N})$. Furthermore, as shown in the proof of \cite[ Theorem 3.1]{HLW}, we have
\begin{equation}\label{trans2}
  \sum_{j\in\mathbb{Z}}\Big\|{\sup_{1\leq k\leq N}}^+(\mathcal{A}_{k}\psi_J)(j)\Big\|_{L_p(\mathcal{N})}^p\leq \Big\|{\sup_{1\leq k\leq N}}^+\mathcal{A}_{k}\psi_J\Big\|_{L_p(L_\infty(\mathbb{Z})\overline{\otimes}\mathcal{N})}^p.
\end{equation}
Using \cite[Proposition 7.3]{HJX}, we obtain the following crucial inequality
\begin{equation}\label{trans3}
  \Big\|{\sup_{1\leq k\leq N}}^+A_k(x)\Big\|_{L_p(\mathcal{N})}\lesssim \Big\|{\sup_{1\leq k\leq N}}^+A_k(\gamma^jx)\Big\|_{L_p(\mathcal{N})}.
\end{equation}
By combining the previously established inequalities \eqref{trans4}, \eqref{trans1}, \eqref{trans2}, and \eqref{trans3}, we can derive the following estimate
\begin{align*}
  (J-N^t) \Big\|{\sup_{1\leq k\leq N}}^+A_k(x)\Big\|_{L_p(\mathcal{N})}^p&\lesssim\sum_{j=1}^{J-N^t} \Big\|{\sup_{1\leq k\leq N}}^+A_k(\gamma^jx)\Big\|_{L_p(\mathcal{N})}^p\\
  &=\sum_{j=1}^{J-N^t} \Big\|{\sup_{1\leq k\leq N}}^+(\mathcal{A}_{k}\psi_J)(j)\Big\|_{L_p(\mathcal{N})}^p\\
  &\lesssim J\|x\|_{L_p(\mathcal{N})}^p.
\end{align*}
From this, we can deduce that
\begin{equation*}
  \frac{J-N^t}{J}\Big\|{\sup_{1\leq k\leq N_0}}^+ {A}_k(x)\Big\|_{L_p(\mathcal{N})}^p\lesssim\|x\|_{L_p(\mathcal{N})}^p.
\end{equation*}
Since $J$ is an arbitrary positive integer, taking the limit as $J\rightarrow\infty$ yields
\begin{equation*}
  \Big\|{\sup_{1\leq k\leq N}}^+A_k(x)\Big\|_{L_p(\mathcal{N})}\lesssim\|x\|_{L_p(\mathcal{N})}.
\end{equation*}
This completes the proof of inequality \eqref{shiftlast}, thereby establishing Theorem \ref{thm1}.

In the following, we concentrate on establishing the proof of inequality \eqref{thm11}.
Assuming, without loss of generality, that $f$ is positive, we note that for $N_2>N_1$, one has
\begin{equation*}
 \mathcal{A}_{N_1}f= f*K_{N_1}\leq\frac{N_2}{N_1}(f*K_{N_2})=\frac{N_2}{N_1}\mathcal{A}_{N_2}f.
\end{equation*}
Consequently, it suffices to demonstrate
\begin{equation}\label{thm12}
   \Big\|{\sup_{k\in\mathbb{N}}}^+\mathcal{A}_{2^k}f\Big\|_{L_p(L_\infty(\mathbb{Z})\overline{\otimes}\mathcal{N})}\leq C_p\|f\|_{L_p(L_\infty(\mathbb{Z})\overline{\otimes}\mathcal{N})}.
\end{equation}
Applying the triangle inequality and noting the facts that for all $m\in\mathbb{N}$,
 \begin{equation*}
   \Big|\Re\Big(f*(K_{2^m}-L_{2^m})\Big)\Big|\leq \Big(\sum_{k=0}^\infty\Big| \Re\Big(f*(K_{2^k}-L_{2^k}\Big)\Big|^p\Big)^{1/p}
 \end{equation*}
 and
 \begin{equation*}
   \Big|\Im\Big(f*(K_{2^m}-L_{2^m})\Big)\Big|\leq \Big(\sum_{k=0}^\infty\Big|\Im\Big(f*(K_{2^k}-L_{2^k}\Big)\Big|^p\Big)^{1/p}.
 \end{equation*}
Then, we have
  \begin{align*}
          &\Big\|{\sup_{k\in\mathbb{N}}}^+f*K_{2^k}\Big\|_{L_p(L_\infty(\mathbb{Z})\overline{\otimes}\mathcal{N})}-\Big\|{\sup_{k\in\mathbb{N}}}^+f*L_{2^k}\Big\|_{L_p(L_\infty(\mathbb{Z})\overline{\otimes}\mathcal{N})}\leq \Big\|{\sup_{k\in\mathbb{N}}}^+f*(K_{2^k}-L_{2^k})\Big\|_{L_p(L_\infty(\mathbb{Z})\overline{\otimes}\mathcal{N})}\\
          &\leq \Big\|{\sup_{k\in\mathbb{N}}}^+\Re\Big(f*(K_{2^k}-L_{2^k})\Big)\Big\|_{L_p(L_\infty(\mathbb{Z})\overline{\otimes}\mathcal{N})}+\Big\|{\sup_{k\in\mathbb{N}}}^+\Im\Big(f*(K_{2^k}-L_{2^k})\Big)\Big\|_{L_p(L_\infty(\mathbb{Z})\overline{\otimes}\mathcal{N})}\\
           &\leq \Big(\sum_{k=0}^\infty\Big\|\Re\Big(f*(K_{2^k}-L_{2^k})\Big)\Big\|_{L_p(L_\infty(\mathbb{Z})\overline{\otimes}\mathcal{N})}^p\Big)^{1/p} +\Big(\sum_{k=0}^\infty\Big\|\Im\Big(f*(K_{2^k}-L_{2^k})\Big)\Big\|_{L_p(L_\infty(\mathbb{Z})\overline{\otimes}\mathcal{N})}^p\Big)^{1/p}\\
            &\lesssim\Big(\sum_{k=0}^\infty\Big\|f*(K_{2^k}-L_{2^k})\Big\|_{L_p(L_\infty(\mathbb{Z})\overline{\otimes}\mathcal{N})}^p\Big)^{1/p}.
          \end{align*}
Therefore, we deduce
 \begin{equation*}
   \Big\|{\sup_{k\in\mathbb{N}}}^+\mathcal{A}_{2^k}f\Big\|_{L_p(L_\infty(\mathbb{Z})\overline{\otimes}\mathcal{N})}\lesssim\Big(\sum_{k=0}^\infty\Big\|f*(K_{2^k}-L_{2^k})\Big\|_{L_p(L_\infty(\mathbb{Z})\overline{\otimes}\mathcal{N})}^p\Big)^{1/p}
   +\Big\|{\sup_{k\in\mathbb{N}}}^+f*L_{2^k}\Big\|_{L_p(L_\infty(\mathbb{Z})\overline{\otimes}\mathcal{N})}.
 \end{equation*}
Observing that
\begin{equation}\label{finaleq}
\min_{\rho\in(0,2)}\max\Big\{\frac{2\rho+1}{\rho+1}, \frac{3\rho+2}{2\rho+1}\Big\}=\frac{2\rho+1}{\rho+1}\Big|_{\rho=\frac{1+\sqrt{5}}{2}}=\frac{1+\sqrt{5}}{2},
\end{equation}
 we note that for $p\in(\frac{1+\sqrt{5}}{2},2]$, there exists a suitable $\rho\in(0,2)$ such that $p>\max\{\frac{2\rho+1}{\rho+1}, \frac{3\rho+2}{2\rho+1}\}$.
Finally, by invoking Lemma \ref{key1} and Lemma \ref{key2}, we arrive at the desired inequality \eqref{thm12}. Additionally, inequality \eqref{thm12} holds for $p=\infty$ trivially. Therefore, through interpolation, we deduce that inequality \eqref{thm12} holds for all $p>\frac{1+\sqrt{5}}{2}$.

\noindent {\bf Acknowledgements} \
G. Hong and L. Wang were supported by National Natural Science Foundation of China (No. 12071355, No. 12325105, No. 12031004).
\bibliographystyle{amsplain}

\begin{thebibliography}{99}	
\bibitem{Ak} M. A. Akcoglu, \textit{A pointwise ergodic theorem in $L_p$-spaces.}  Canadian J. Math. \textbf{27} (1975), no. 5, 1075-1082.
\bibitem{CAD} C. Anantharaman-Delaroche, \textit{On ergodic theorems for free group actions on noncommutative spaces.}  Probab. Theory Related Fields \textbf{135} (2006), no. 4, 520-546.
\bibitem{Be}T. Bekjan, \textit{Noncommutative maximal ergodic theorems for positive contractions.}   J. Funct. Anal. \textbf{254} (2008), no. 9,  2401-2418.
\bibitem{Bo881}J. Bourgain, \textit{On the maximal ergodic theorem for certain subsets of the integers.}   Israel J. Math. \textbf{61} (1988), no. 1, 39-72.
\bibitem{Bo882}J. Bourgain, \textit{On the pointwise ergodic theorem on $L^p$ for arithmetic sets. }  Israel J. Math. \textbf{61} (1988), no. 1, 73-84.
\bibitem{Bo89}J. Bourgain, \textit{Pointwise ergodic theorems for arithmetic sets. With an appendix by the author, Harry Furstenberg, Yitzhak Katznelson and Donald S. Ornstein.}  Inst. Hautes \'{E}tudes Sci. Publ. Math. No. 69 (1989), 5-45.
%\bibitem{CXY}Z. Chen, Q. Xu and Z. Yin, \textit{Harmonic Analysis on Quantum Tori.} Commun. Math. Phys. \textbf{322} (2013), 755-805.
\bibitem{CW}L. Cadilhac, S. Wang, \textit{Noncommutative maximal ergodic inequalities for amenable groups.}    arXiv preprint arXiv:2206.12228.
\bibitem{CoN} J. P. Conze, N. Dang-Ngoc, \textit{Ergodic theorems for noncommutative dynamical systems.} Invent. Math. \textbf{46} (1978), no. 1, 1-15.
\bibitem{DS} N. Dunford, J. T. Schwartz, \textit{Convergence almost everywhere of operator averages.} J. Rational Mech. Anal. \textbf{5} (1956), 129-178.
\bibitem{FITW20} C. Fefferman, A. Ionescu, T. Tao, S. Wainger et al., \textit{Analysis and Applications: The mathematical work of Elias Stein.} Bull. Amer. Math. Soc. (N.S.) 57 (2020), 523-594.




\bibitem{HJX}U. Haagerup, M. Junge, and Q. Xu, \textit{A reduction method for noncommutative $L_p$-spaces and applications.}  Trans. Amer. Math. Soc. \textbf{362} (2010), no. 4, 2125-2165.
\bibitem{Ho}G. Hong, \textit{ Non-commutative ergodic averages of balls and spheres over Euclidean spaces.} Ergodic Theory Dynam. Systems \textbf{40} (2020), no. 2, 418-436.
\bibitem{HLSX} G. Hong, X. Lai, S. Ray, and B. Xu, \textit{Noncommutative maximal strong $L_p$ estimates of Calder\'{o}n-Zygmund operators.}    arXiv preprint arXiv:2212.13150 . To appear in Israel. J. Math.


\bibitem{HLW} G. Hong, B. Liao, and S. Wang, \textit{Noncommutative maximal ergodic inequalities associated with doubling conditions. }    Duke Math. J. \textbf{170} (2021), no. 2, 205-246.
\bibitem{HRW} G. Hong, S. Ray, and S. Wang, \textit{Maximal ergodic inequalities for some positive operators on noncommutative $L_p$-spaces.}  J. Lond. Math. Soc. (2) \textbf{108} (2023), no. 1, 362-408.
\bibitem{HoSu18} G. Hong, M. Sun, \textit{Noncommutative multi-parameter Wiener-Wintner type ergodic theorem.} J. Funct. Anal. 275 (2018), 1100-1137.
\bibitem{Hu} Y. Hu, \textit{Maximal ergodic theorems for some group actions.}  J. Funct. Anal. \textbf{254} (2008), no. 5, 1282-1306.
\bibitem{JX07}M. Junge, Q. Xu, \textit{Noncommutative maximal ergodic theorems.} J. Amer. Math. Soc. \textbf{20} (2007) 385-439.
\bibitem{KMT22} B. Krause, M. Mirek and T. Tao, \textit{Pointwise ergodic theorems for non-conventional bilinear polynomial averages.} Ann. of Math. (2) 195 (3) (2022) 997-1109.



\bibitem{Ku}B. K\"{u}mmerer, \textit{A non-commutative individual ergodic theorem.}  Invent. Math. \textbf{46} (1978), no. 2, 139-145.
\bibitem{La}E. C. Lance, \textit{Ergodic theorems for convex sets and operator algebras.} Invent. Math. \textbf{37} (1976), no. 3, 201-214.
\bibitem{LeXu12} C. Le Merdy, Q. Xu, \textit{Maximal theorems and square functions for analytic operators on $L_p$-spaces.} J. London Math. Soc. 86 (2012), 343-365.
\bibitem{Li}E. Lindenstrauss, \textit{Pointwise theorems for amenable groups.} Invent. Math. \textbf{146} (2001), no. 2, 259-295.
\bibitem{LS}S. Litvinov, \textit{A non-commutative Wiener-Wintner theorem.}  Illinois J. Math. \textbf{58} (2014), no. 3, 697-708.
\bibitem{MSW}A. Magyar, E. M. Stein, and S. Wainger, \textit{Discrete analogues in harmonic analysis: spherical averages.}  Ann. of Math. (2) \textbf{155} (2002), no. 1, 189-208.
\bibitem{Mei09}T. Mei, \textit{Operator valued Hardy spaces.} Mem. Amer. Math. Soc. \textbf{188} (2007) vi+64 pp.

\bibitem{Nev06} A. Nevo, \textit{Chapter 13 - Pointwise Ergodic Theorem For Actions For Groups, Handbook of Dynamical Systems.} Elsevier Science 1 (2006), 871-982.

\bibitem{PX03}G. Pisier, Q. Xu, \textit{Noncommutative $L^p$ spaces.} Handbook of Geometry of Banach spaces, Vol. 2 (2003), 1459-1517.
\bibitem{Wi} N. Wiener, \textit{The ergodic theorem.} Duke Math. J. \textbf{5} (1939), no. 1, 1-18.
\bibitem{Ye}F. J. Yeadon, \textit{Ergodic theorems for semifinite von Neumann algebras. I. }  J. London Math. Soc. (2) \textbf{16} (1977), no. 2, 326-332.
%\bibitem{bo}J. Bourgain, \textit{Besicovitch-type maximal operators and applications to Fourier analysis.} Geom. Funct. Anal. \textbf{21} (1991), 147-187.
%\bibitem{BD}J. Bourgain, C. Demeter, \textit{The proof of the $\ell^2$ decoupling conjecture.}  Ann. of Math. (2) \textbf{182} (2015), no.1, 351-389.
%\bibitem{CXY}Z. Chen, Q. Xu and Z. Yin, \textit{Harmonic Analysis on Quantum Tori.} Commun. Math. Phys. \textbf{322} (2013), 755--805.
%\bibitem{De} C. Demeter, \textit{Fourier restriction, decoupling, and applications.} Cambridge Studies in Advanced Mathematics, \textbf{184}. Cambridge University Press, Cambridge, 2020. xvi+331 pp.
%\bibitem{DN01}  M. Douglas, N. Nekrasov, \textit{ Noncommutative field theory.} Rev. Modern Phys. \textbf{71} (2001), no. 4, 977-1029.
%\bibitem{Fe1} C. Fefferman, \textit{Inequalities for strongly singular convolution operators.} Acta Math. \textbf{124} (1970), 9-36.
%\bibitem{Fe2} C. Fefferman, A. Ionescu, T. Tao and S. Wainger, \textit{Analysis and applications: the mathematical work of Elias Stein. With contributions from Loredana Lanzani, Akos Magyar, Mariusz Mirek, et al.}  Bull. Amer. Math. Soc. (N.S.) \textbf{57} (2020), no. 4, 523-594.
%\bibitem{GJM} L. Gao, M. Junge and E. McDonald, \textit{Quantum Euclidean spaces with noncommutative derivatives.} \textbf{16} J. Noncommut. Geom. (2022), no.1, 153-213.
%\bibitem{GJP}A. Gonz\'{a}lez-P\'{e}rez, M. Junge and J. Parcet, \textit{Singular integrals in quantum Euclidean spaces.} Mem. Amer. Math. Soc. \textbf{272} (2021), no.1334.
%\bibitem{GF}L. Grafakos, \textit{Modern Fourier Analysis.} Third edition. Graduate Texts in Mathematics, \textbf{250}. Springer, New York, 2014. xvi+624 pp.
%\bibitem{Guth} L. Guth, J. Hickman and M. Iliopoulou,  \textit{Sharp estimates for oscillatory integral operators via polynomial partitioning.} Acta Math. \textbf{223} (2019), no. 2, 251-376.
%\bibitem{HZ} J. Hickman, J. Zahl, \textit{A note on Fourier restriction and nested Polynomial Wolff axioms.}    arXiv preprint arXiv:2010.02251, 2020.
%\bibitem{JPMX} M. Junge, J. Parcet, T. Mei and R. Xia, \textit{Algebraic Calder\'{o}n-Zygmund theory.} Adv. Math. \textbf{376} (2021), Paper No. 107443, 72 pp.
%\bibitem{Lai21} X. Lai, \textit{Sharp estimates of noncommutative Bochner-Riesz means on two-dimensional quantum tori.} Commun. Math. Phys. \textbf{390} (2022), no. 1, 193-230.
%\bibitem{LSZ} G. Levitina, F. Sukochev and D. Zanin, \textit{Cwikel estimates revisited.} Proc. Lond. Math. Soc. (3) \textbf{120} (2020), no. 2, 265-304.
%\bibitem{Xiong1} E. McDonald, F. Sukochev and X. Xiong, \textit{Quantum differentiability on noncommutative Euclidean spaces.} Commun. Math. Phys. \textbf{379} (2020), 491-542.
%\bibitem{MS} C. Muscalu, W. Schlag, \textit{Classical and multilinear harmonic analysis.} Vol. I. Cambridge Studies in Advanced Mathematics, 137. Cambridge University Press, Cambridge, 2013. xviii+370 pp.
%\bibitem{NS98} N. Nekrasov, A. Schwarz, \textit{Instantons on noncommutative $\mathbb{R}^4$, and $(2,0)$ superconformal six-dimensional theory.} Commun. Math. Phys. \textbf{198} (1998), no. 3, 689-703.
%    \bibitem{Pi}G. Pisier, \textit{Introduction to operator space theory.}   London Mathematical Society Lecture Note Series, \textbf{294}. Cambridge University Press, Cambridge, 2003.
%\bibitem{PX03}G. Pisier, Q. Xu, \textit{Noncommutative $L^p$ spaces.} Handbook of Geometry of Banach spaces, Vol. 2 (2003), 1459--1517.
%	%\bibitem{stein} Stein E M, Murphy T S, \textit{ Harmonic analysis: real-variable methods, orthogonality, and oscillatory integrals}. Princeton University Press, 1993.
%\bibitem{Ro} K. M. Rogers, \textit{A local smoothing estimate for the Schr\"{o}dinger equation.}  Adv. Math. \textbf{219} (2008), no. 6, 2105-2122.
%\bibitem{SW99} N. Seiberg, E. Witten, \textit{String theory and noncommutative geometry.}  J. High Energy Phys. (1999), no. 9, Paper 32, 93 pp.
%\bibitem{St}E. M. Stein, \textit{Interpolation of linear operators.}  Trans. Amer. Math. Soc. \textbf{86} (1956), 482-492.
%\bibitem{Str}R. S. Strichartz, \textit{Restrictions of Fourier transforms to quadratic surfaces and decay of solutions of wave equations. }  Duke Math. J. \textbf{44} (1977), no. 3, 705-714.
%\bibitem{To1}P. A. Tomas, \textit{A restriction theorem for the Fourier transform.}  Bull. Amer. Math. Soc. \textbf{81} (1975), 477-478.
%\bibitem{To2}P. A. Tomas, \textit{Restriction theorems for the Fourier transform.}  Harmonic analysis in Euclidean spaces (Proc. Sympos. Pure Math., Williams Coll., Williamstown, Mass., 1978), Proc. Sympos. Pure Math., XXXV, Part, Amer. Math. Soc., Providence, R.I., 1979, pp. 111-114.
%\bibitem{Tao}T. Tao, \textit{Some recent progress on the restriction conjecture.} Fourier analysis and convexity, 217-243, Appl. Numer. Harmon. Anal., Birkh\"{a}user Boston, Boston, MA, 2004.
%\bibitem{Tao03}T. Tao, \textit{A sharp bilinear restrictions estimate for paraboloids.}   Geom. Funct. Anal. \textbf{13} (2003), no. 6, 1359-1384.
%\bibitem{Vg}J. Voigt, \textit{Abstract Stein interpolation.}   Math. Nachr. \textbf{157} (1992), 197-199.
%\bibitem{Xiong18}X. Xiong, Q. Xu and Z. Yin, \textit{Sobolev, Besov and Triebel-Lizorkin Spaces on
%Quantum Tori.} Mem. Amer. Math. Soc. \textbf{252} (2018).
%\bibitem{Wang}H. Wang, \textit{A restriction estimate in $\mathbb{R}^3$ using brooms.} Duke Math. J. \textbf{171} (2022), no.8, 1749-1822.
%	\bibitem{Zy1}A. Zygmund, \textit{On Fourier coefficients and transforms of functions of two variables.}   Studia Math. \textbf{50} (1974), 189-201.
	
	
	
\end{thebibliography}

\end{document}